\documentclass{amsart}
\usepackage{amsfonts,graphicx,rotating,ctable}
\usepackage{amssymb}

\makeatletter
\@namedef{subjclassname@2020}{\textup{2020} Mathematics Subject Classification}
\makeatother

\newcommand{\cupp}{\operatorname{cup}}
\newcommand{\height}{\operatorname{ht}}
\newcommand{\cat}{\operatorname{cat}}

\newtheorem{theorem}{Theorem}[section]
\newtheorem{lemma}[theorem]{Lemma}

\newtheorem{proposition}[theorem]{Proposition}

\theoremstyle{definition}
\newtheorem{definition}[theorem]{Definition}
\newtheorem{example}[theorem]{Example}

\theoremstyle{remark}
\newtheorem{remark}[theorem]{Remark}

\numberwithin{equation}{section}


\begin{document}

\title[Cup-length of oriented Grassmann manifolds via Gr\"obner bases]{Cup-length of oriented Grassmann manifolds via Gr\"obner bases}

\author{Uro\v s A.\ Colovi\'c}
\address{University of Belgrade,
  Faculty of mathematics,
  Studentski trg 16,
  Belgrade,
  Serbia}
\email{mm21033@alas.matf.bg.ac.rs}

\author{Branislav I.\ Prvulovi\'c}
\address{University of Belgrade,
  Faculty of mathematics,
  Studentski trg 16,
  Belgrade,
  Serbia}
\email{bane@matf.bg.ac.rs}
\thanks{The second author was partially supported by the Ministry of Science, Technological Development and Innovations of the Republic of Serbia [contract no.\ 451-03-47/2023-01/200104].}





\begin{abstract}
The aim of this paper is to prove a conjecture made by T.\ Fukaya in 2008. This conjecture concerns the exact value of the $\mathbb Z_2$-cup-length of the Grassmann manifold $\widetilde G_{n,3}$ of oriented $3$-planes in $\mathbb R^n$. Along the way, we calculate the heights of the Stiefel--Whitney classes of the canonical vector bundle over $\widetilde G_{n,3}$.
\end{abstract}

\maketitle



\section{Introduction}
\label{intro}

For a space $X$ and a commutative ring $R$, the \em $R$-cup-length \em of $X$, denoted by $\cupp_R(X)$, is defined as the supremum of the set of all integers $m$ with the property that there exist positive dimensional (not necessarily mutually different) cohomo\-logy classes $x_1,\ldots,x_m\in H^*(X;R)$ such that their cup product $x_1\cup\cdots\cup x_m$ is nonzero. Although this invariant is relevant and interesting in its own right, perhaps the main reason for studying and calculating $\cupp_R(X)$ is the fact that it provides a lower bound for the Lyusternik--Shnirel'man category $\cat(X)$ of the space $X$ (defined as the minimal $d$ with the property that $X$ can be covered with $d$ open subsets each of which is contractible in $X$). Namely, it is well known that $\cat(X)\ge1+\cupp_R(X)$ for any commutative ring $R$. Computing the Lyusternik--Shnirel'man categories of Lie groups, homogeneous spaces, and some other commonly used spaces, is a difficult and longstanding problem in topology \cite{Ganea}.

When it comes to Grassmann manifolds $G_{n,k}$ of $k$-dimensional subspaces in $\mathbb R^n$, the most notable work on their $\mathbb Z_2$-cup-length was done by Berstein \cite{Berstein}, Hiller \cite{Hiller} and Stong \cite{Stong}. In \cite{Stong} one can find the exact value of $\cupp_{\mathbb Z_2}(G_{n,k})$ for $k\le4$ and all $n$, as well as the exact value of $\cupp_{\mathbb Z_2}(G_{2^t+1,5})$ for $t\ge3$.

The computation of $\cupp_{\mathbb Z_2}(\widetilde G_{n,k})$, where $\widetilde G_{n,k}$ is the Grassmann manifold of \em oriented \em $k$-dimensional subspaces in $\mathbb R^n$, is more challenging because the cohomology algebra $H^*(\widetilde G_{n,k};\mathbb Z_2)$ is more complicated than $H^*(G_{n,k};\mathbb Z_2)$. Moreover, there is no complete description of this algebra in general. Since $\widetilde G_{n,1}=S^{n-1}$, the case $k=1$ is trivial: $\cupp_{\mathbb Z_2}(\widetilde G_{n,1})=1$ for all $n\ge2$ (it suffices to study Grassmannians $\widetilde G_{n,k}$ with $n\ge2k$, because $\widetilde G_{n,k}\cong\widetilde G_{n,n-k}$). Also, it is known that $\cupp_{\mathbb Z_2}(\widetilde G_{n,2})=\lfloor n/2\rfloor$ ($n\ge4$) \cite[Theorem 3.6]{KorbasRusin:ChRank}.

In the case $k=3$, a conjecture on the value of $\cupp_{\mathbb Z_2}(\widetilde G_{n,3})$ was presented by Fukaya in \cite{Fukaya}. In that paper, he proves the conjecture for the case $n=2^t-1$ ($t\ge3$), and this result was independently obtained by Korba\v s in \cite{Korbas:Cup}. In the meantime, the conjecture was confirmed in the cases $n\in\{2^t,2^t+1,2^t+2\}$ ($t\ge3$) by Korba\v s and Rusin \cite{Korbas:ChRank,KorbasRusin:ChRank}; in the cases $n\in\{2^t+2^{t-1}+1,2^t+2^{t-1}+2\}$ ($t\ge3$) by Rusin \cite{Rusin:ChRank}; and the former result was generalized to all $n$ such that $2^t-1\le n\le2^t-1+2^t/3$ ($t\ge3$) in \cite{PPR:ChRank}.

The main result of this paper is the following theorem, which resolves positively Fukaya's conjecture for all $n$ (\cite[Conjecture 1.2]{Fukaya}). The proof we present is self-contained and encompasses the previously known cases.

\begin{theorem}\label{thm1}
Let $n\ge7$ be a fixed integer. If $t\ge3$ is the integer with the property $2^t-1\le n<2^{t+1}-1$, then
\[\cupp_{\mathbb Z_2}(\widetilde G_{n,3})=\begin{cases}
2^t-3,& 2^t-1\le n\le2^t+2^{t-1}-2\\
2^t-2,& n=2^t+2^{t-1}-1\\
2^t-1,& n=2^t+2^{t-1}\\
n-2^s-1,& 2^{t+1}-2^{s+1}+1\le n\le2^{t+1}-2^s \quad (1\le s\le t-2)\\
\end{cases}.\]
\end{theorem}

In Table \ref{table:1} we give a schematic display of dependence of $\cupp_{\mathbb Z_2}(\widetilde G_{n,3})$ on $n$ in the range specified in the theorem. The second group of rows in the table corresponds to the case $s=t-2$, while the last group (the last two rows) corresponds to the case $s=1$.

\begin{table}[h!]
\scriptsize
\centering
\begin{tabular}{||m{1.9cm} m{1.8cm} m{3.4cm} m{1.7cm} m{1.8cm}||}
 \hline
 $n$ & $\cupp_{\mathbb Z_2}(\widetilde G_{n,3})$ & a monomial which realizes $\cupp_{\mathbb Z_2}(\widetilde G_{n,3})$ & $\height(\widetilde w_2)$ & $\height(\widetilde w_3)$ \\ [0.5ex]
 \hline\hline
 $2^t-1$ & $2^t-3$ & $\widetilde w_2^{2^t-4}a_{2^t-4}$ & $2^t-4$ & $2^{t-1}-2$\\
 \hline
  $2^t$ & $2^t-3$ & $\widetilde w_2^{2^t-4}a_{2^t-1}$ & $2^t-4$ & $2^{t-1}-2$\\
 \hline
  $2^t+1$ & $2^t-3$ & $\widetilde w_2^{2^t-4}a_{3n-2^{t+1}-1}$ & $2^t-4$ & $2^{t-1}-2$\\
 \hline
 $\cdot$ & $\cdot$ & $\cdot$ & $\cdot$ & $\cdot$\\
 $\cdot$ & $\cdot$ & $\cdot$ & $\cdot$ & $\cdot$\\
 $\cdot$ & $\cdot$ & $\cdot$ & $\cdot$ & $\cdot$\\
 \hline
  $2^t+2^{t-1}-2$ & $2^t-3$ & $\widetilde w_2^{2^t-4}a_{3n-2^{t+1}-1}$ & $2^t-4$ & $2^{t-1}-2$\\
 \hline
  $2^t+2^{t-1}-1$ & $2^t-2$ & $\widetilde w_2^{2^{t-1}-1}\widetilde w_3^{2^{t-1}-2}a_{2^{t+1}-4}$ & $2^t-4$ & $2^{t-1}-2$\\
\hline
  $2^t+2^{t-1}$ & $2^t-1$ & $\widetilde w_2^{2^{t-1}-1}\widetilde w_3^{2^{t-1}-1}a_{2^{t+1}-4}$ & $2^t-4$ & $2^{t-1}-1$\\
\hline\hline
  $2^t+2^{t-1}+1$ & $2^t+2^{t-2}$ & $\widetilde w_2^{2^t+2^{t-2}-1}a_{2^{t+1}-4}$ & $2^t+2^{t-2}-1$ & $2^{t-1}$\\
\hline
  $2^t+2^{t-1}+2$ & $2^t+2^{t-2}+1$ & $\widetilde w_2^{2^t+2^{t-2}-1}\widetilde w_3a_{2^{t+1}-4}$ & $2^t+2^{t-2}-1$ & $2^{t-1}+1$\\
\hline
$\cdot$ & $\cdot$ & $\cdot$ & $\cdot$ & $\cdot$\\
$\cdot$ & $\cdot$ & $\cdot$ & $\cdot$ & $\cdot$\\
$\cdot$ & $\cdot$ & $\cdot$ & $\cdot$ & $\cdot$\\
\hline
  $2^t+2^{t-1}+2^{t-2}$ & $2^t+2^{t-1}-1$ & $\widetilde w_2^{2^t+2^{t-2}-1}\widetilde w_3^{2^{t-2}-1}a_{2^{t+1}-4}$ & $2^t+2^{t-2}-1$ & $2^{t-1}+2^{t-2}-1$\\
\hline\hline
 $\cdot$ & $\cdot$ & $\cdot$ & $\cdot$ & $\cdot$\\
 $\cdot$ & $\cdot$ & $\cdot$ & $\cdot$ & $\cdot$\\
 $\cdot$ & $\cdot$ & $\cdot$ & $\cdot$ & $\cdot$\\
 \hline\hline
$2^{t+1}-2^{s+1}+1$ & $2^{t+1}-3\cdot2^s$ & $\widetilde w_2^{2^{t+1}-3\cdot2^s-1}a_{2^{t+1}-4}$ & $2^{t+1}-3\cdot2^s-1$ & $2^t-2^{s+1}$\\
\hline
$2^{t+1}-2^{s+1}+2$ & $2^{t+1}-3\cdot2^s+1$ & $\widetilde w_2^{2^{t+1}-3\cdot2^s-1}\widetilde w_3a_{2^{t+1}-4}$ & $2^{t+1}-3\cdot2^s-1$ & $2^t-2^{s+1}+1$\\
\hline
$\cdot$ & $\cdot$ & $\cdot$ & $\cdot$ & $\cdot$\\
$\cdot$ & $\cdot$ & $\cdot$ & $\cdot$ & $\cdot$\\
$\cdot$ & $\cdot$ & $\cdot$ & $\cdot$ & $\cdot$\\
\hline
$2^{t+1}-2^s$ & $2^{t+1}-2^{s+1}-1$ & $\widetilde w_2^{2^{t+1}-3\cdot2^s-1}\widetilde w_3^{2^s-1}a_{2^{t+1}-4}$ & $2^{t+1}-3\cdot2^s-1$ & $2^t-2^s-1$\\
\hline\hline
 $\cdot$ & $\cdot$ & $\cdot$ & $\cdot$ & $\cdot$\\
 $\cdot$ & $\cdot$ & $\cdot$ & $\cdot$ & $\cdot$\\
 $\cdot$ & $\cdot$ & $\cdot$ & $\cdot$ & $\cdot$\\
 \hline\hline
$2^{t+1}-3$ & $2^{t+1}-6$ & $\widetilde w_2^{2^{t+1}-7}a_{2^{t+1}-4}$ & $2^{t+1}-7$ & $2^t-4$\\
\hline
$2^{t+1}-2$ & $2^{t+1}-5$ & $\widetilde w_2^{2^{t+1}-7}\widetilde w_3a_{2^{t+1}-4}$ & $2^{t+1}-7$ & $2^t-3$\\
\hline
\end{tabular}
\caption{}
\label{table:1}
\end{table}

Let us remark here that the notation in \cite{Fukaya} (and likewise in \cite{PPR:ChRank}) is slightly diffe\-rent than the one used in this paper. Here, $\widetilde G_{n,3}$ consists of oriented $3$-dimensional subspaces of $\mathbb R^n$, while in \cite{Fukaya} $\widetilde G_{n,3}$ denotes the manifold of oriented $3$-dimensional subspaces of $\mathbb R^{n+3}$. So, one should make this adjustment when comparing Theorem \ref{thm1} with Conjecture 1.2 from \cite{Fukaya}.

Intimately connected with cup-length is the notion of the height of a cohomology class. For such a class $x$, its \em height \em, denoted by $\height(x)$, is defined as the supremum of the set of all integers $m$ with the property $x^m\neq0$. Let $\widetilde w_i\in H^i(\widetilde G_{n,3};\mathbb Z_2)$, $i=2,3$, be the Stiefel--Whitney classes of the canonical ($3$-dimensional) vector bundle $\widetilde\gamma_{n,3}$ over $\widetilde G_{n,3}$. In this paper we compute the heights of these Stiefel--Whitney classes for all values of $n$ (cf.\ Table \ref{table:1}).

\begin{theorem}\label{thm2}
Let $n\ge7$ be a fixed integer. If $t\ge3$ is the integer with the property $2^t-1\le n<2^{t+1}-1$, then
\[\height(\widetilde w_2)=\begin{cases}
2^t-4,& 2^t-1\le n\le2^t+2^{t-1}\\
2^{t+1}-3\cdot2^s-1,& 2^{t+1}-2^{s+1}+1\le n\le2^{t+1}-2^s \quad (1\le s\le t-2)\\
\end{cases}.\]
\end{theorem}

\begin{theorem}\label{thm3}
Let $n\ge7$ be a fixed integer. If $t\ge3$ is the integer with the property $2^t-1\le n<2^{t+1}-1$, then
\[\height(\widetilde w_3)=\max\{2^{t-1}-2,n-2^t-1\}.\]
\end{theorem}

Note that all three theorems omit the case $n=6$. However, in this particular case there is a description of the cohomology algebra $H^*(\widetilde G_{n,3};\mathbb Z_2)$ \cite[Proposition 3.1(1)]{KorbasRusin:Palermo}, from which it is readily seen that $\cupp_{\mathbb Z_2}(\widetilde G_{6,3})=3$ and $\height(\widetilde w_2)=\height(\widetilde w_3)=1$ (so this case fits into Table \ref{table:1} as its last row for $t=2$).

The classes $a_m$ (for various $m$) appearing in Table \ref{table:1} are the so-called "anomalous" or "indecomposable" classes (other than $\widetilde w_2$ and $\widetilde w_3$) in $H^*(\widetilde G_{n,3};\mathbb Z_2)$ detected by Basu and Chakraborty in \cite[Theorem A]{BasuChakraborty}.

The main tool for proving Theorems \ref{thm1}--\ref{thm3} will be the theory of Gr\"obner bases. The subalgebra of $H^*(\widetilde G_{n,3};\mathbb Z_2)$ generated by $\widetilde w_2$ and $\widetilde w_3$ is known to be isomorphic to the quotient of the polynomial algebra $\mathbb Z_2[w_2,w_3]$ by an ideal $I_n$. This ideal is generated by well-known polynomials $g_{n-2},g_{n-1},g_n\in\mathbb Z_2[w_2,w_3]$. In Theorem \ref{Grebner}, for an arbitrary $n$, we detect a Gr\"obner basis for $I_n$, which allows us to perform some nontrivial calculations in $H^*(\widetilde G_{n,3};\mathbb Z_2)$. This Gr\"obner basis for $I_n$ turns out to be a generalization of the Gr\"obner basis obtained by Fukaya \cite{Fukaya} in the special case $n=2^t-1$ (actually, the ideal $I_{2^t-1}$ happens to coincide with $I_{2^t}$, and so, Fukaya's basis cover the case $n=2^t$ as well).

The paper is organized as follows. In Section \ref{prel} we collect some preliminary facts concerning the cohomology algebra $H^*(\widetilde G_{n,3};\mathbb Z_2)$ and its subalgebra generated by $\widetilde w_2$ and $\widetilde w_3$. In Section \ref{Grobner}, after a very brief introduction to the theory of Gr\"obner bases over the field $\mathbb Z_2$ , we exhibit a set of polynomials in $\mathbb Z_2[w_2,w_3]$, and eventually prove that it is a Gr\"obner basis for the ideal $I_n$. Section \ref{heights} is devoted to computing the heights of $\widetilde w_2$ and $\widetilde w_3$. Both computations use the Gr\"obner basis obtained in Section \ref{Grobner}. We first prove Theorem \ref{thm3}, and then perform a considerable amount of calculation in the polynomial algebra $\mathbb Z_2[w_2,w_3]$ in order to prove Theorem \ref{thm2}. Finally, in Section \ref{cup-length} we apply the results from previous two sections and prove Theorem \ref{thm1}.

In the rest of the paper the $\mathbb Z_2$ coefficients for cohomology will be understood, and so we will abbreviate $H^*(\widetilde G_{n,3};\mathbb Z_2)$ to $H^*(\widetilde G_{n,3})$.

\section{Background on cohomology algebra $H^*(\widetilde G_{n,3})$}
\label{prel}

Let $n\ge7$ be an integer and $W_n$ the subalgebra of the cohomology algebra $H^*(\widetilde G_{n,3})$ generated by the Stiefel--Whitney classes $\widetilde w_2$ and $\widetilde w_3$. It is well known (see e.g.\ \cite{Fukaya}) that
\begin{equation}\label{isomorphism}
W_n\cong\frac{\mathbb Z_2[ w_2, w_3]}{(g_{n-2},g_{n-1},g_n)},
\end{equation}
where $g_r\in\mathbb Z_2[w_2,w_3]$ is the homogeneous polynomial (of degree $r$) obtained from the equation
\[(1+w_2+w_3)(g_0+g_1+g_2+\cdots)=1\]
(it is understood that the degree of $w_i$ is $i$). It is obvious that $g_0=1$, $g_1=0$, $g_2=w_2$, and that the following recurrence formula holds:

\begin{equation}\label{recgpolk3}
g_{r+3}=w_2g_{r+1}+w_3g_r \quad \mbox{for all } r\ge0.
\end{equation}
Now it is easy to calculate a few of these polynomials. In Table \ref{table:2} we list polynomials $g_r$ for $0\le r\le25$.

\begin{table}[h!]
\footnotesize
\centering
\begin{tabular}{||m{0.9cm} m{3cm} ||}
\hline
 $r$ & $g_r$ \\ [0.5ex]
\hline\hline
 $0$ & $1$\\
\hline
 $1$ & $0$\\
\hline
 $2$ & $w_2$\\
\hline
 $3$ & $w_3$\\
\hline
 $4$ & $w_2^2$\\
\hline
 $5$ & $0$\\
\hline
 $6$ & $w_2^3+w_3^2$\\
\hline
 $7$ & $w_2^2w_3$\\
\hline
 $8$ & $w_2^4+w_2w_3^2$\\
\hline
 $9$ & $w_3^3$\\
\hline
 $10$ & $w_2^5$\\
\hline
 $11$ & $w_2^4w_3$\\
\hline
 $12$ & $w_2^6+w_3^4$\\
\hline
 $13$ & $0$\\
\hline
 $14$ & $w_2^7+w_2^4w_3^2+w_2w_3^4$\\
\hline
 $15$ & $w_2^6w_3+w_3^5$\\
\hline
 $16$ & $w_2^8+w_2^5w_3^2+w_2^2w_3^4$\\
\hline
 $17$ & $w_2^4w_3^3$\\
\hline
 $18$ & $w_2^9+w_2^3w_3^4+w_3^6$\\
\hline
 $19$ & $w_2^8w_3+w_2^2w_3^5$\\
\hline
 $20$ & $w_2^{10}+w_2w_3^6$\\
\hline
 $21$ & $w_3^7$\\
\hline
 $22$ & $w_2^{11}+w_2^8w_3^2$\\
\hline
 $23$ & $w_2^{10}w_3$\\
\hline
 $24$ & $w_2^{12}+w_2^9w_3^2+w_3^8$\\
\hline
 $25$ & $w_2^8w_3^3$\\
\hline
\end{tabular}
\caption{}
\label{table:2}
\end{table}

For convenience, denote the ideal $(g_{n-2},g_{n-1},g_n)$ by $I_n$. By (\ref{recgpolk3}) we see that actually
\begin{equation}\label{uskladiti}
g_r\in I_n \quad \mbox{ for all } r\ge n-2.
\end{equation}
An obvious consequence of (\ref{uskladiti}) is the fact that the sequence of ideals $\{I_n\}_{n\ge2}$ is descending:
\[I_n\supseteq I_{n+1} \quad \mbox{ for all } n\ge2.\]

Via the isomorphism (\ref{isomorphism}) the class of $w_i$ in the quotient $\mathbb Z_2[w_2,w_3]/I_n$ corresponds to the Stiefel--Whitney class $\widetilde w_i\in H^i(\widetilde G_{n,3})$, $i=2,3$ (this is the reason why the grading on $\mathbb Z_2[w_2,w_3]$ is such that the degree of $w_i$ is $i$). So the following equivalence (which we use throughout the paper) holds for all nonnegative integers $b$ and $c$:
\begin{equation}\label{ekv}
w_2^bw_3^c\in I_n \quad \Longleftrightarrow \quad \widetilde w_2^b\widetilde w_3^c=0 \mbox{ in } H^*(\widetilde G_{n,3}).
\end{equation}

The identity (\ref{recgpolk3}) can easily be generalized. One can use induction on $j$ to prove that
\begin{equation}\label{lemspol1}
g_{r+3\cdot2^j}=w_2^{2^j}g_{r+2^j}+w_3^{2^j}g_r, \quad r,j\ge0.
\end{equation}
(see \cite[(2.4)]{Rusin:ChRank}).

The following lemma will be repeatedly used throughout the paper.

\begin{lemma}\label{kvadriranje}
For all nonnegative integers $r$ we have
\[w_3g_r^2=g_{2r+3}.\]
\end{lemma}
\begin{proof}
The proof is by induction on $r$. It is obvious from Table \ref{table:2} that the lemma is true for $0\le r\le2$. Now, if $r\ge3$ and the lemma is true for $r-3$ and $r-2$, then by (\ref{lemspol1}) we have
\begin{align*}
w_3g_r^2&=w_3(w_2g_{r-2}+w_3g_{r-3})^2=w_3(w_2^2g_{r-2}^2+w_3^2g_{r-3}^2)=w_2^2w_3g_{r-2}^2+w_3^3g_{r-3}^2\\
        &=w_2^2g_{2r-1}+w_3^2g_{2r-3}=g_{2r+3},
\end{align*}
and the induction is completed.
\end{proof}

As the first usage of this lemma, we single out and compute a few types of polynomials $g_r$.

\begin{proposition}\label{g-3}
Let $t\ge2$ be an integer. Then:
\begin{itemize}
\item[(a)] $g_{2^t-3}=0$;
\item[(b)] $g_{2^t+2^{t-1}-3}=w_3^{2^{t-1}-1}$;
\item[(c)] $g_{2^t+2^{t-2}-3}=w_2^{2^{t-2}}w_3^{2^{t-2}-1}$;
\item[(d)] $g_{2^t+2^{t-1}+2^{t-2}-3}=w_2^{2^{t-1}}w_3^{2^{t-2}-1}$;
\item[(e)] $g_{2^t+2^{t-1}+2^{t-3}-3}=w_2^{2^{t-1}+2^{t-3}}w_3^{2^{t-3}-1}$ (if $t\ge3$).
\end{itemize}
\end{proposition}
\begin{proof}
The proofs of all equalities are by induction on $t$. From Table \ref{table:2} we see that the identities (a)--(d) are true for $t=2$, and that (e) holds for $t=3$. It is now routine to apply Lemma \ref{kvadriranje} and complete the induction step. For example,
\[g_{2^{t+1}+2^t+2^{t-1}-3}=w_3(g_{2^t+2^{t-1}+2^{t-2}-3})^2=w_3w_2^{2^t}w_3^{2^{t-1}-2}=w_2^{2^t}w_3^{2^{t-1}-1},\]
which proves (d).
\end{proof}

As another application of Lemma \ref{kvadriranje}, let us give a simple proof of the equality
\begin{equation}\label{kvadrat}
(g_{2^{t-1}-2})^2=g_{2^t-4},
\end{equation}
which holds for all integers $t\ge2$ (cf.\ \cite[Lemma 2.3]{Rusin:ChRank}). We have
\[w_3(g_{2^{t-1}-2})^2=g_{2^t-1}=w_3g_{2^t-4}+w_2g_{2^t-3}=w_3g_{2^t-4}\]
(by (\ref{recgpolk3}) and Proposition \ref{g-3}(a)), and since canceling is allowed in $\mathbb Z_2[w_2,w_3]$, we have established (\ref{kvadrat}).

\section{Gr\"obner bases}
\label{Grobner}

\subsection{Background on Gr\"obner bases}

 The theory of Gr\"obner bases has been well established for decades. It has proved itself as a valuable tool in dealing with ideals of the polynomial rings. In what follows we give some basic preliminaries from this theory, but we confine ourselves to the polynomial ring $\mathbb{Z}_2[w_2,w_3]$ (since we are going to work in $\mathbb{Z}_2[w_2,w_3]$ only). Actually, we will only define a few notions and cite a theorem that we need for the subsequent parts of the paper. A comprehensive treatment of the theory of Gr\"obner bases the reader can find in \cite{Becker}.

        The set of all monomials in $\mathbb{Z}_2[w_2,w_2]$ will be denoted by $M$. Let $\preceq$ be a well order
        on $M$ such that $m_1\preceq m_2$ implies $mm_1\preceq mm_2$ for any $m,m_1,m_2\in M$.
        For a nonzero polynomial $p=\sum_i m_i\in\mathbb Z_2[w_2,w_3]$, where $m_i$ are pairwise different monomials,
        define its \em leading monomial \em $\mathrm{LM}(p)$ as $\max_im_i$ with respect to $\preceq$.

        This already suffices for a definition of Gr\"obner basis.

\begin{definition}
    Let $F\subset\mathbb Z_2[w_2,w_3]$ be a finite set of nonzero polynomials and $I$ an ideal in $\mathbb Z_2[w_2,w_3]$. $F$ is a \em Gr\"obner basis for $I$ \em if $I$ is generated by $F$ and for any $p\in I\setminus\{0\}$ there exists $f\in F$ such that
    $\mathrm{LM}(f)\mid\mathrm{LM}(p)$.
\end{definition}

The notion of an $m$-representation (for a monomial $m\in M$) will be important to us.
    For a finite set of nonzero polynomials $F\subset \mathbb{Z}_2[w_2,w_3]$ and a monomial $m\in M$, we say that
    \[ p=\sum_{i=1}^km_if_i\]
    is an \em $m$-representation \em of a nonzero polynomial $p$ \em with respect to \em $F$ if $m_1,\ldots,m_k\in M$, $f_1,\ldots,f_k\in F$ and
    $\mathrm{LM}(m_if_i)\preceq m$ for every $i\in\{1,\ldots,k\}$ (note that it is not required for $f_i$'s to be pairwise different).

    For two monomials $m_1,m_2\in M$, we denote their least common multiply by $\mathrm{lcm}(m_1,m_2)$. If $p,q\in\mathbb{Z}_2[w_2,w_3]$ are nonzero polynomials, their \em $S$-polynomial \em is defined as:
    \begin{equation}\label{Spol}
    S(p,q)=\frac{\mathrm{lcm}\big(\mathrm{LM}(p),\mathrm{LM}(q)\big)}{\mathrm{LM}(p)}\cdot p+
        \frac{\mathrm{lcm}\big(\mathrm{LM}(p),\mathrm{LM}(q)\big)}{\mathrm{LM}(q)}\cdot q.
    \end{equation}
Note that $S(p,p)=0$ and $S(q,p)=S(p,q)$.

    The following theorem gives us a sufficient condition for a set of polynomials to be a Gr\"obner basis.

\begin{theorem}\label{baker}
    A finite set of nonzero polynomials $F\subset \mathbb{Z}_2[w_2,w_3]$ that generate an ideal $I$, is a Gr\"obner basis for $I$, if for every $f,g\in F$ we have that
    $S(f,g)$ either equals zero or has an $m$-representation with $m\prec\mathrm{lcm}\big(\mathrm{LM}(f),\mathrm{LM}(g)\big)$.
\end{theorem}

    A proof of this theorem can be found in \cite[Theorem 5.64]{Becker}.

\subsection{Gr\"obner basis for the ideal $I_n\trianglelefteq\mathbb Z_2[w_2,w_3]$}

We first fix a monomial order $\preceq$ in $\mathbb Z_2[w_2,w_3]$. We will be using lexicographic monomial ordering with $w_3\prec w_2$, that is
\[w_2^{b_1} w_3^{c_1}\preceq w_2^{b_2} w_3^{c_2} \qquad \Longleftrightarrow \qquad b_1<b_2 \quad \vee \quad (b_1=b_2 \quad \wedge \quad c_1\leq c_2).\]

Let $n\ge7$ be a fixed integer. We are looking for a Gr\"obner basis for the ideal $I_n=(g_{n-2},g_{n-1},g_n)$. If $t\ge3$ is the integer such that $2^t-1\le n<2^{t+1}-1$, we are going to work with the binary expansion of the number $n-2^t+1$:
\[n-2^t+1=\sum_{j=0}^{t-1}\alpha_j2^j.\]
We denote by $s_i$ the $i$-th partial sum $\sum_{j=0}^i\alpha_j2^j$ ($0\le i\le t-1$), and we also define $s_{-1}:=0$. Observe now the polynomials
\begin{equation}\label{fpol}
f_i=w_3^{\alpha_is_{i-1}}g_{n-2+2^i-s_i}, \quad 0\le i\le t-1.
\end{equation}
We are going to prove that $F:=\{f_0,f_1,\ldots,f_{t-1}\}$ is a Gr\"obner basis for the ideal $I_n$.

Since $s_{t-1}=n-2^t+1$, we have that
\begin{equation}\label{ft-1}
f_{t-1}=w_3^{\alpha_{t-1}s_{t-2}}g_{2^t+2^{t-1}-3}=w_3^{\alpha_{t-1}s_{t-2}+2^{t-1}-1},
\end{equation}
by Proposition \ref{g-3}(b).

Let us now compute explicitly the last two polynomials from $F$ in a few cases that will be relevant in our upcoming calculations (for $f_{t-1}$ we use (\ref{ft-1})).

\begin{example}\label{2t-1}
If $\underline{n=2^t-1}$, then $n-2^t+1=0$, and so $\alpha_i=s_i=0$ for all $i$. Therefore, by Proposition \ref{g-3}(c),
\begin{align*}
f_{t-2}&=g_{2^t-1-2+2^{t-2}}=g_{2^t+2^{t-2}-3}=w_2^{2^{t-2}}w_3^{2^{t-2}-1},\\
f_{t-1}&=w_3^{2^{t-1}-1}.
\end{align*}
\end{example}

In the cases $n=2^t+2^{t-1}-1$ and $n=2^t+2^{t-1}$ we shall need the last three polynomials from $F$.

\begin{example}\label{2t+2t-1-1}
In the case $\underline{n=2^t+2^{t-1}-1}$ we have $n-2^t+1=2^{t-1}$, which implies $\alpha_{t-3}=\alpha_{t-2}=0$ and $s_{t-3}=s_{t-2}=0$. Now we use Proposition \ref{g-3}(d,e) to calculate:
\begin{align*}
f_{t-3}&=g_{2^t+2^{t-1}-1-2+2^{t-3}}=g_{2^t+2^{t-1}+2^{t-3}-3}=w_2^{2^{t-1}+2^{t-3}}w_3^{2^{t-3}-1},\\
f_{t-2}&=g_{2^t+2^{t-1}-1-2+2^{t-2}}=g_{2^t+2^{t-1}+2^{t-2}-3}=w_2^{2^{t-1}}w_3^{2^{t-2}-1},\\
f_{t-1}&=w_3^{2^{t-1}-1}.
\end{align*}
\end{example}

\begin{example}\label{2t+2t-1}
If $\underline{n=2^t+2^{t-1}}$, then $n-2^t+1=1+2^{t-1}$, which means that $\alpha_{t-3}s_{t-4}=0$, $\alpha_{t-2}=0$, $\alpha_{t-1}=1$ and $s_{t-3}=s_{t-2}=1$. Therefore,
\begin{align*}
f_{t-3}&=g_{2^t+2^{t-1}-2+2^{t-3}-1}=g_{2^t+2^{t-1}+2^{t-3}-3}=w_2^{2^{t-1}+2^{t-3}}w_3^{2^{t-3}-1},\\
f_{t-2}&=g_{2^t+2^{t-1}-2+2^{t-2}-1}=g_{2^t+2^{t-1}+2^{t-2}-3}=w_2^{2^{t-1}}w_3^{2^{t-2}-1},\\
f_{t-1}&=w_3^{1+2^{t-1}-1}=w_3^{2^{t-1}}.
\end{align*}
\end{example}

\begin{example}\label{2t+1-2s+1+1}
Now let $\underline{n=2^{t+1}-2^{s+1}+1}$ for some $s\in\{1,2,\ldots,t-2\}$. Then we have $n-2^t+1=2^t-2^{s+1}+2=2+2^{s+1}+\cdots+2^{t-1}$, and conclude that $\alpha_{t-1}=1$, $s_{t-2}=2^{t-1}-2^{s+1}+2$. Using Proposition \ref{g-3}(d), (\ref{fpol}) and (\ref{ft-1}) we get
\begin{align*}
f_{t-2}&=w_3^{\alpha_{t-2}s_{t-3}}g_{2^{t+1}-2^{t-2}-3}\!=\!w_3^{\alpha_{t-2}s_{t-3}}g_{2^t+2^{t-1}+2^{t-2}-3}=w_2^{2^{t-1}}w_3^{\alpha_{t-2}s_{t-3}+2^{t-2}-1},\\
f_{t-1}&=w_3^{2^{t-1}-2^{s+1}+2+2^{t-1}-1}=w_3^{2^t-2^{s+1}+1}.
\end{align*}
\end{example}

\begin{example}\label{2t+1-2s}
If $\underline{n=2^{t+1}-2^s}$ for some $s\in\{1,2,\ldots,t-2\}$, then $n-2^t+1=2^t-2^s+1=1+2^s+\cdots+2^{t-1}$, which implies $\alpha_{t-2}=\alpha_{t-1}=1$, $s_{t-2}=2^{t-1}-2^s+1$, $s_{t-3}=2^{t-2}-2^s+1$, and so
\begin{align*}
f_{t-2}&=w_3^{2^{t-2}-2^s+1}g_{2^t+2^{t-1}+2^{t-2}-3}=w_2^{2^{t-1}}w_3^{2^{t-2}-2^s+1+2^{t-2}-1}=w_2^{2^{t-1}}w_3^{2^{t-1}-2^s},\\
f_{t-1}&=w_3^{2^{t-1}-2^s+1+2^{t-1}-1}=w_3^{2^t-2^s}.
\end{align*}
\end{example}

Let us now determine the leading monomials of the polynomials from $F$. Actually, we are able to calculate $\mathrm{LM}(g_r)$ for all $r$ for which this makes sense (i.e., for which $g_r\neq0$). If $r+3$ is a power of two, then we know that $g_r=0$ (Proposition \ref{g-3}(a)). If $r+3$ is not a power of two, then there exist unique integers $i,l\ge0$ such that $r+3=2^i(2l+3)$. The next lemma deals with this (nontrivial) case.

\begin{lemma}\label{vodeci}
Let $i$ and $l$ be nonnegative integers. Then $g_{2^i(2l+3)-3}\neq0$ and
\[\mathrm{LM}(g_{2^i(2l+3)-3})=w_2^{2^il}w_3^{2^i-1}.\]
\end{lemma}
\begin{proof}
We can prove this lemma by induction on $i$. For $i=0$ we need to prove
\begin{equation}\label{bazaind}
\mathrm{LM}(g_{2l})=w_2^l \quad \mbox{ for all } l\ge0.
 \end{equation}
The monomial $w_2^l$ is the greatest monomial (with respect to $\preceq$) in degree $2l$, and since $g_{2l}$ is homogeneous of this degree, it suffices to show that $w_2^l$ appears in $g_{2l}$ with nonzero coefficient. This is obviously true for small values of $l$ (see Table \ref{table:2}), and so the induction on $l$ and the identity $g_{2l}=w_2g_{2l-2}+w_3g_{2l-3}$ finishes the proof of (\ref{bazaind}).

For the induction step, assume that $i\ge1$, $l\ge0$ and that $\mathrm{LM}(g_{2^{i-1}(2l+3)-3})=w_2^{2^{i-1}l}w_3^{2^{i-1}-1}$. Then by Lemma \ref{kvadriranje} we have
\[
\mathrm{LM}(g_{2^i(2l+3)-3})=\mathrm{LM}\big(w_3(g_{2^{i-1}(2l+3)-3})^2\big)=w_3(w_2^{2^{i-1}l}w_3^{2^{i-1}-1})^2=w_2^{2^il}w_3^{2^i-1},
\]
and the proof is complete.
\end{proof}

\begin{proposition}\label{LMgi}
For $i\in\{0,1,\ldots,t-1\}$ we have $f_i\neq0$ and
\[\mathrm{LM}(f_i)=w_2^{\frac{n+1-s_i}{2}-2^i}w_3^{\alpha_is_{i-1}+2^i-1}.\]
\end{proposition}
\begin{proof}
Since $f_i=w_3^{\alpha_is_{i-1}}g_{n-2+2^i-s_i}$, we have $\mathrm{LM}(f_i)=w_3^{\alpha_is_{i-1}}\mathrm{LM}(g_{n-2+2^i-s_i})$. Therefore, if we write $n-2+2^i-s_i$ in the form $2^i(2l+3)-3$ (for some $l\ge0$), we will be able to apply the previous lemma and thus compute $\mathrm{LM}(f_i)$.

Note first that $n+1-s_i$ is divisible by $2$. Namely, \begin{align*}
n+1-s_i&=n+1-\bigg(s_{t-1}-\sum_{j=i+1}^{t-1}\alpha_j2^j\bigg)=n+1-\bigg(n-2^t+1-\sum_{j=i+1}^{t-1}\alpha_j2^j\bigg)\\
&=2^t+\sum_{j=i+1}^{t-1}\alpha_j2^j
\end{align*}
(it is understood that this sum is zero if $i=t-1$). Now we have
\begin{align*}
        n-2+2^i-s_i&=2^t+\sum_{j=i+1}^{t-1}\alpha_j2^j+2^i-3=2^i\bigg(2^{t-i}+\sum_{j=i+1}^{t-1}\alpha_j2^{j-i}+1\bigg)-3\\
                   &=2^i\bigg(2\Big(2^{t-1-i}+\sum_{j=i+1}^{t-1}\alpha_j2^{j-i-1}-1\Big)+3\bigg)-3.
\end{align*}
So, we apply Lemma \ref{vodeci} for $l=2^{t-1-i}+\sum_{j=i+1}^{t-1}\alpha_j2^{j-i-1}-1$, and since
\[2^il=2^{t-1}+\sum_{j=i+1}^{t-1}\alpha_j2^{j-1}-2^i=\frac{2^t+\sum_{j=i+1}^{t-1}\alpha_j2^j}{2}-2^i=\frac{n+1-s_i}{2}-2^i,\]
we are done.
\end{proof}

We will also use the following notation. For a monomial $m=w_2^bw_3^c$:
\[\deg_{w_2}(m):=b, \quad \deg_{w_3}(m):=c.\]
So the degree of $m$ (in the chosen grading in $\mathbb Z_2[w_2,w_3]$) is $2\deg_{w_2}(m)+3\deg_{w_3}(m)$.

Since $s_i$ and $2^i$ increase with $i$, note that Proposition \ref{LMgi} implies
\begin{equation}\label{LMgiopada}
\deg_{w_2}\big(\mathrm{LM}(f_i)\big)>\deg_{w_2}\big(\mathrm{LM}(f_{i+1})\big), \qquad 0\leq i\leq t-2.
\end{equation}

The next lemma shows that certain polynomials which naturally appear in upcoming calculations are actually elements of $F$ (if they are nonzero).
It will help us in proving Proposition \ref{jednakost ideala} and Lemma \ref{uzastopni}. Recall that $\alpha_0,\alpha_1,\ldots,\alpha_{t-1}$ are binary digits of the number $n-2^t+1$ and $s_i=\sum_{j=0}^i\alpha_j2^j$ ($-1\le i\le t-1$).

\begin{lemma}\label{glavna1}
Let $i\in\{0,1,\ldots,t-1\}$ be an integer.
\begin{itemize}
\item[(a)] Let $z\in\{i,i+1,\ldots,t-1\}$ be the smallest integer with the property $\alpha_z=0$ (i.e., $\alpha_i=\cdots=\alpha_{z-1}=1$ and $\alpha_z=0$), if such an integer exists. Then
       \[
        g_{n-2+2^i-s_{i-1}}=
        \begin{cases}
           f_z, & \text{if $z$ exists}\\
           0, & \text{otherwise}
        \end{cases}.
        \]
\item[(b)] Let $u\in\{i,i+1,\ldots,t-1\}$ be the smallest integer with the property $\alpha_u=1$ (i.e., $\alpha_i=\cdots=\alpha_{u-1}=0$ and $\alpha_u=1$), if such an integer exists. Then
    \[
        w_3^{s_{i-1}}g_{n-2-s_{i-1}}=
        \begin{cases}
           f_u, & \text{if $u$ exists}\\
           0, & \text{otherwise}
        \end{cases}.
        \]
\end{itemize}
\end{lemma}
\begin{proof}
    (a) If $z$ does not exist, i.e., if $\alpha_i=\cdots=\alpha_{t-1}=1$, then $n-2^t+1=s_{t-1}=s_{i-1}+\sum_{j=i}^{t-1}2^j=s_{i-1}+2^t-2^i$, and so $n-2+2^i-s_{i-1}=2^{t+1}-3$, which means that $g_{n-2+2^i-s_{i-1}}=0$ (Proposition \ref{g-3}(a)).

    Suppose now that $z$ exists. Then $s_z-s_{i-1}=\sum_{j=i}^{z-1}2^j=2^z-2^i$, and so $n-2+2^i-s_{i-1}=n-2+2^z-s_z$, which implies:
    \[
    f_z=w_3^{\alpha_zs_{z-1}}g_{n-2+2^z-s_z}= g_{n-2+2^i-s_{i-1}}.
    \]

    (b) Similarly as in part (a), if $u$ does not exist, i.e., if $\alpha_i=\cdots=\alpha_{t-1}=0$, then $s_{i-1}=s_{t-1}=n-2^t+1$, leading to $g_{n-2-s_{i-1}}=g_{2^t-3}=0$.

    Suppose that $u$ exists. Then $\alpha_us_{u-1}=s_{u-1}=s_{i-1}$, as well as $s_u=s_{i-1}+2^u$, and so $n-2-s_{i-1}=n-2+2^u-s_u$. Therefore
    \[
        f_u=w_3^{\alpha_us_{u-1}}g_{n-2+2^u-s_u}= w_3^{s_{i-1}}g_{n-2-s_{i-1}}.
    \]
    This concludes the proof.
\end{proof}

In order to prove that the set $F=\{f_0,f_1,\ldots,f_{t-1}\}$ is a Gr\"obner basis for $I_n$, we have to show that $F$ generates the ideal $I_n$, which is the statement of the following proposition.

\begin{proposition}\label{jednakost ideala}
    We have the equality of ideals: $I_n=(F)$.
\end{proposition}
\begin{proof}
    We first prove $(F)\subseteq I_n$. By (\ref{uskladiti}) we know that $g_{n-2+\nu}\in I_n$ for all nonnegative integers $\nu$. Let us prove (by induction on $\nu$) that
    \begin{equation}\label{pripadnost}
    w_3^\nu g_{n-2-\nu}\in I_n, \quad 0\le\nu\le n-2.
    \end{equation}
The base case $\nu=0$ is trivial, so suppose $\nu\ge1$. Then by (\ref{recgpolk3}) we have
    \begin{align*}
        w_3^\nu g_{n-2-\nu}&=w_3^{\nu-1}w_3g_{n-2-\nu}=w_3^{\nu-1}(w_2g_{n-2-\nu+1}+g_{n-2-\nu+3})\\
                    &=w_2w_3^{\nu-1}g_{n-2-(\nu-1)}+w_3^{\nu-1}g_{n-2-(\nu-3)},
    \end{align*}
    and from inductive hypothesis, both summands are in $I_n$, which completes the proof of (\ref{pripadnost}).

    Now we show that $f_i\in I_n$ for $0\le i\le t-1$. If $\alpha_i=0$, then $s_i=s_{i-1}<2^i$, and so $f_i=g_{n-2+2^i-s_i}\in I_n$. If $\alpha_i=1$, then $s_i=s_{i-1}+2^i$, and so $f_i=w_3^{s_{i-1}}g_{n-2-s_{i-1}}\in I_n$ by (\ref{pripadnost}). Therefore, $(F)=(f_0,f_1,\ldots,f_{t-1})\subseteq I_n$.

   For the reverse containment it is enough to prove $g_{n-2},g_{n-1},g_n\in(F)$:
    \begin{itemize}
        \item The relation $g_{n-2}\in(F)$ follows immediately from Lemma \ref{glavna1}(b) for $i=0$.
        \item Similarly, $g_{n-1}\in(F)$ is obtained from Lemma \ref{glavna1}(a) for $i=0$.
        \item If $\alpha_0=0$, i.e., $s_0=0$, then Lemma \ref{glavna1}(a) applied to $i=1$ implies $g_n\in(F)$. If $\alpha_0=1$, then we can apply Lemma \ref{glavna1}(b) for $i=1$, and obtain $w_3g_{n-3}\in(F)$. We have already proved that $g_{n-2}\in(F)$, and so equation (\ref{recgpolk3}) gives us
        \[g_n=w_3g_{n-3}+w_2g_{n-2}\in(F).\]
    \end{itemize}
This concludes the proof of the proposition.
\end{proof}

We will prove that $F$ is a Gr\"obner basis by using Theorem \ref{baker}. Therefore, we want to calculate $S$-polynomials of polynomials from $F$. For $0\le i<j\le t-1$, we already know that $\deg_{w_2}\big(\mathrm{LM}(f_i)\big)>\deg_{w_2}\big(\mathrm{LM}(f_j)\big)$ (see (\ref{LMgiopada})), and since $\alpha_is_{i-1}+2^i\le s_{i-1}+2^i<2^j$, we have
\[\deg_{w_3}\big(\mathrm{LM}(f_i)\big)=\alpha_is_{i-1}+2^i-1<2^j-1\le\alpha_js_{j-1}+2^j-1=\deg_{w_3}\big(\mathrm{LM}(f_j)\big)\]
(see Proposition \ref{LMgi}). This means that
\begin{equation}\label{lcm}
\mathrm{lcm}\big(\mathrm{LM}(f_i),\mathrm{LM}(f_j)\big)=w_2^{\frac{n+1-s_i}{2}-2^i}w_3^{\alpha_js_{j-1}+2^j-1}, \qquad 0\le i\le j\le t-1.
\end{equation}
Now by (\ref{Spol}) and Proposition \ref{LMgi} we have
\begin{equation}\label{Spolinomij}
S(f_i,f_j)=w_3^{\alpha_js_{j-1}-\alpha_is_{i-1}+2^j-2^i}f_i+w_2^{\frac{s_j-s_i}{2}+2^j-2^i}f_j, \qquad 0\le i\le j\le t-1.
\end{equation}

Now we turn to proving that these $S$-polynomials indeed have appropriate representations. We will do this inductively: representations of $S(f_i,f_j)$ and $S(f_j,f_{j+1})$ will give us a desired representation of $S(f_i,f_{j+1})$. The following lemma establishes a relation between these polynomials.

\begin{lemma}\label{rekurentna}
    For $0\le i\le j\le t-2$ the following identity holds:
    \[S(f_i,f_{j+1})=w_3^{\alpha_{j+1}s_j-\alpha_js_{j-1}+2^j}S(f_i,f_j)+w_2^{\frac{s_j-s_i}{2}+2^j-2^i}S(f_j,f_{j+1}).\]
\end{lemma}
\begin{proof}
We use (\ref{Spolinomij}) and just calculate:
    \begin{align*}
        S(f_i,f_{j+1})=&\,w_3^{\alpha_{j+1}s_j-\alpha_is_{i-1}+2^{j+1}-2^i}f_i+w_2^{\frac{s_{j+1}-s_i}{2}+2^{j+1}-2^i}f_{j+1}\\
        =&\,w_3^{\alpha_{j+1}s_j-\alpha_is_{i-1}+2^{j+1}-2^i}f_i+w_2^{\frac{s_j-s_i}{2}+2^j-2^i}w_3^{\alpha_{j+1}s_j-\alpha_js_{j-1}+2^j}f_j\\
         &+w_2^{\frac{s_j-s_i}{2}+2^j-2^i}w_3^{\alpha_{j+1}s_j-\alpha_js_{j-1}+2^j}f_j+w_2^{\frac{s_{j+1}-s_i}{2}+2^{j+1}-2^i}f_{j+1}\\
        =&\,w_3^{\alpha_{j+1}s_j-\alpha_js_{j-1}+2^j}\left(w_3^{\alpha_js_{j-1}-\alpha_is_{i-1}+2^j-2^i}f_i+w_2^{\frac{s_j-s_i}{2}+2^j-2^i}f_j\right) \\
        &+w_2^{\frac{s_j-s_i}{2}+2^j-2^i}\left(w_3^{\alpha_{j+1}s_j-\alpha_js_{j-1}+2^j}f_j+w_2^{\frac{s_{j+1}-s_j}{2}+2^j}f_{j+1}\right) \\
        =&\,w_3^{\alpha_{j+1}s_j-\alpha_js_{j-1}+2^j}S(f_i,f_j)+w_2^{\frac{s_j-s_i}{2}+2^j-2^i}S(f_j,f_{j+1}),
    \end{align*}
and we are done.
\end{proof}

We first deal with the case of consecutive polynomials from $F$.

\begin{lemma}\label{uzastopni}
   For $0\le j\le t-2$ the polynomial $S(f_j,f_{j+1})$ is either zero or has an $m$-representation with respect to $F$ such that
 \[\mathrm{deg}_{w_2}(m)<\frac{n+1-s_j}{2}-2^j.\]
\end{lemma}

\begin{proof}
First we use (\ref{Spolinomij}) and (\ref{fpol}) to calculate:
    \begin{align*}
        S(f_j,f_{j+1})&=w_3^{\alpha_{j+1}s_j-\alpha_js_{j-1}+2^j}f_j+w_2^{\frac{s_{j+1}-s_j}{2}+2^j}f_{j+1}\\
                      &=w_3^{\alpha_{j+1}s_j+2^j}g_{n-2+2^j-s_j}+w_2^{\alpha_{j+1}2^j+2^j}w_3^{\alpha_{j+1}s_j}g_{n-2+2^{j+1}-s_{j+1}}\\
                      &=w_3^{\alpha_{j+1}s_j}\left(w_3^{2^j}g_{n-2-s_j+2^j}+w_2^{\alpha_{j+1}2^j+2^j}g_{n-2-s_{j+1}+2^{j+1}}\right).
    \end{align*}
We now distinguish two cases. Suppose first that $\alpha_{j+1}=0$. Then $s_j=s_{j+1}$, and using (\ref{lemspol1}) we have:
\[S(f_j,f_{j+1})=w_3^{2^j}g_{n-2-s_{j+1}+2^j}+w_2^{2^j}g_{n-2-s_{j+1}+2^{j+1}}=g_{n-2-s_{j+1}+2^{j+2}}.\]
If $j=t-2$, then $s_{j+1}=s_{t-1}=n-2^t+1$, and so $S(f_j,f_{j+1})=g_{2^{t+1}-3}=0$ (Proposition \ref{g-3}(a)).
If $j\le t-3$, then from Lemma \ref{glavna1}(a) applied to $i=j+2$, it follows that $g_{n-2-s_{j+1}+2^{j+2}}$ is equal to either zero or $f_z$ for some $z\in\{j+2,\ldots,t-1\}$.
In the latter case, we have an $m$-representation of $S(f_j,f_{j+1})$ with respect to $F$, where $m=\mathrm{LM}(f_z)$, and so
\[\deg_{w_2}(m)\le\deg_{w_2}\big(\mathrm{LM}(f_{j+2})\big)<\deg_{w_2}\big(\mathrm{LM}(f_j)\big)=\frac{n+1-s_j}{2}-2^j\]
by (\ref{LMgiopada}) and Proposition \ref{LMgi}. This completes the proof in the first case.

Next, suppose $\alpha_{j+1}=1$. Then $s_j+2^{j+1}=s_{j+1}$, and after two applications of (\ref{lemspol1}) we get:
    \begin{align*}
        S(f_j,f_{j+1})=&\,w_3^{s_j}\left(w_3^{2^j}g_{n-2-s_j+2^j}+w_2^{2^{j+1}}g_{n-2-s_{j+1}+2^{j+1}}\right)\\
        =&\,w_3^{s_j}\!\left(\!w_2^{2^j}g_{n-2-s_j+2^{j+1}}\!+\!g_{n-2-s_j+2^{j+2}}\!+w_3^{2^{j+1}}\!g_{n-2-s_{j+1}}\!\!+\!g_{n-2-s_j+2^{j+2}}\!\right)\\
        =&\,w_2^{2^j}w_3^{s_j}g_{n-2-s_j+2^{j+1}}+w_3^{s_{j+1}}g_{n-2-s_{j+1}}.
    \end{align*}
In order to show that this leads to a representation of $S(f_j,f_{j+1})$ we are looking for, it is now enough to express these two summands (actually, those which are nonzero) in the form $\overline mf$, where $\overline m\in M$ and $f\in F$ are such that $\deg_{w_2}\big(\mathrm{LM}(\overline mf)\big)<(n+1-s_j)/2-2^j$.

If the first summand is nonzero, we apply Lemma \ref{glavna1}(a) for $i=j+1$ to conclude that there is $z\in\{j+1,\ldots,t-1\}$ such that $g_{n-2-s_j+2^{j+1}}=f_z$, and then
\begin{align*}
\deg_{w_2}\!\big(\mathrm{LM}(w_2^{2^j}w_3^{s_j}f_z)\big)&=\deg_{w_2}\!\big(w_2^{2^j}w_3^{s_j}\mathrm{LM}(f_z)\big)=2^j+\deg_{w_2}\!\big(\mathrm{LM}(f_z)\big)\\
&\le2^j+\deg_{w_2}\big(\mathrm{LM}(f_{j+1})\big)\\
&=2^j+\frac{n+1-s_{j+1}}{2}-2^{j+1}<\frac{n+1-s_j}{2}-2^j.
\end{align*}

If the second summand is nonzero, then $j\le t-3$ ($j=t-2$ implies $s_{j+1}=s_{t-1}=n-2^t+1$, leading to $w_3^{s_{j+1}}g_{n-2-s_{j+1}}=w_3^{s_{j+1}}g_{2^t-3}=0$), and so we can apply Lemma \ref{glavna1}(b) for $i=j+2$ to conclude that $w_3^{s_{j+1}}g_{n-2-s_{j+1}}=f_u$ for some $u\in\{j+2,\ldots,t-1\}$. Moreover,
\[
\deg_{w_2}\big(\mathrm{LM}(f_u)\big)\le\deg_{w_2}\big(\mathrm{LM}(f_{j+2})\big)<\deg_{w_2}\big(\mathrm{LM}(f_j)\big)=\frac{n+1-s_j}{2}-2^j,
\]
which finishes the proof.
\end{proof}

Finally, we can use the previous two lemmas to get an appropriate representation for $S(f_i,f_j)$ and thus prove the main theorem of this section.

\begin{theorem}\label{Grebner}
The set $F=\{f_0,f_1,\ldots,f_{t-1}\}$ (see (\ref{fpol})) is a Gr\"obner basis for $I_n$ with respect to $\preceq$.
\end{theorem}
\begin{proof}
Let $i$ and $j$ be integers such that $0\le i\le j\le t-1$. We shall prove the following claim:
\begin{itemize}
\item $S(f_i,f_j)$ is either zero or has an $m$-representation with respect to $F$ such that
\[\mathrm{deg}_{w_2}(m)<\frac{n+1-s_i}{2}-2^i.\]
\end{itemize}
By Theorem \ref{baker}, this will prove the theorem, because then
\[m\prec w_2^{\frac{n+1-s_i}{2}-2^i}w_3^{\alpha_js_{j-1}+2^j-1}=\mathrm{lcm}\big(\mathrm{LM}(f_i),\mathrm{LM}(f_j)\big)\]
(see (\ref{lcm})).

To prove the claim, we fix $i$ and work by induction on $j$. Since $S(f_i,f_i)=0$, the base case $j=i$ is trivial. Now we assume that the claim is true for an integer $j$ such that $i\le j\le t-2$, and prove that it is true for $j+1$ as well. By Lemma \ref{rekurentna},
\begin{equation}\label{Spolinomij+1}
S(f_i,f_{j+1})=w_3^{\alpha_{j+1}s_j-\alpha_js_{j-1}+2^j}S(f_i,f_j)+w_2^{\frac{s_j-s_i}{2}+2^j-2^i}S(f_j,f_{j+1}).
\end{equation}
Note that now it suffices to prove for each of these summands that it is either zero or has an $m$-representation with respect to $F$ such that $\mathrm{deg}_{w_2}(m)<(n+1-s_i)/2-2^i$. This is clear if some of these summands is zero, and if both of them are nonzero, just add up the two representations.

For the first summand in (\ref{Spolinomij+1}), we know by induction hypothesis that if $S(f_i,f_j)$ is nonzero, then it has an $\widetilde m$-representation with $\mathrm{deg}_{w_2}(\widetilde m)<(n+1-s_i)/2-2^i$. If we multiply this representation by $w_3^{\alpha_{j+1}s_j-\alpha_js_{j-1}+2^j}$, we obtain an $m$-representation of the first summand, where $m=w_3^{\alpha_{j+1}s_j-\alpha_js_{j-1}+2^j}\cdot\widetilde m$. But then
\[\mathrm{deg}_{w_2}(m)=\mathrm{deg}_{w_2}(\widetilde m)<\frac{n+1-s_i}{2}-2^i.\]
For the second summand in (\ref{Spolinomij+1}) we use Lemma \ref{uzastopni}, which guarantees that if $S(f_j,f_{j+1})\neq0$, then $S(f_j,f_{j+1})$ has an $\widetilde m$-representation (with respect to $F$) with $\mathrm{deg}_{w_2}(\widetilde m)<(n+1-s_j)/2-2^j$. This representation multiplied by $w_2^{(s_j-s_i)/2+2^j-2^i}$ is an $m$-representation of the second summand, where $m=w_2^{(s_j-s_i)/2+2^j-2^i}\cdot\widetilde m$. Now we have
\begin{align*}
\mathrm{deg}_{w_2}(m)&=\frac{s_j-s_i}{2}+2^j-2^i+\mathrm{deg}_{w_2}(\widetilde m)<\frac{s_j-s_i}{2}+2^j-2^i+\frac{n+1-s_j}{2}-2^j\\
                     &=\frac{n+1-s_i}{2}-2^i.
\end{align*}
This completes the proof that $F=\{f_0,f_1,\ldots,f_{t-1}\}$ is a Gr\"obner basis for $I_n$.
\end{proof}

\begin{remark}
We have used Sage in order to compute these Gr\"obner basis for $I_n$ in cases $n\le64$. From that calculation we were able to conjecture how the basis should look like generally, and to successfully prove it afterwards.
\end{remark}

\begin{remark}
In \cite{Fukaya} Fukaya found Gr\"obner bases for the ideals $I_n$ when $n$ is of the form $2^t-1$. It is not hard to check that our Gr\"obner bases coincide with Fukaya's in that case. Moreover, it is readily seen from (\ref{fpol}) that our basis $F$ is the same in the case $n=2^t$ as well. This is no surprise, since $g_{2^t-3}=0$ implies $g_{2^t}=w_2g_{2^t-2}$ (by (\ref{recgpolk3})), and so
\[I_{2^t-1}=(g_{2^t-3},g_{2^t-2},g_{2^t-1})=(g_{2^t-2},g_{2^t-1})=(g_{2^t-2},g_{2^t-1},g_{2^t})=I_{2^t}.\]
\end{remark}

\section{Heights of $\widetilde w_2$ and $\widetilde w_3$}
\label{heights}

\subsection{The height of $\widetilde w_3$}

Having the Gr\"obner basis $F$ we can easily determine the height of the Stiefel--Whitney class $\widetilde w_3$, and thus prove Theorem \ref{thm3}. By (\ref{ekv}) we are actually looking for the integer $d$ with properties: $w_3^d\notin I_n$ and $w_3^{d+1}\in I_n$.

\begin{proof}[Proof of Theorem \ref{thm3}]
According to (\ref{ft-1}), $w_3^{\alpha_{t-1}s_{t-2}+2^{t-1}-1}=f_{t-1}\in I_n$. On the other hand, it is obvious that the monomial $w_3^{\alpha_{t-1}s_{t-2}+2^{t-1}-2}$ is not divisible by $\mathrm{LM}(f_{t-1})$, and it is not divisible by $\mathrm{LM}(f_i)$ for $0\le i\le t-2$ either, because $\mathrm{deg}_{w_2}\big(\mathrm{LM}(f_i)\big)>\mathrm{deg}_{w_2}\big(\mathrm{LM}(f_{t-1})\big)=0$ (by (\ref{LMgiopada})). Since $F=\{f_0,f_1,\ldots,f_{t-1}\}$ is a Gr\"obner basis, $w_3^{\alpha_{t-1}s_{t-2}+2^{t-1}-2}\notin I_n$. This means that $\height(\widetilde w_3)=\alpha_{t-1}s_{t-2}+2^{t-1}-2$, and so we are left to prove
\[\alpha_{t-1}s_{t-2}+2^{t-1}-2=\max\{2^{t-1}-2,n-2^t-1\}.\]

If $\alpha_{t-1}=0$, this amounts to proving the inequality $2^{t-1}-2\ge n-2^t-1$. But in this case $n-2^t+1=\sum_{j=0}^{t-1}\alpha_j2^j=\sum_{j=0}^{t-2}\alpha_j2^j<2^{t-1}$, and we are done.

If $\alpha_{t-1}=1$, then $\alpha_{t-1}s_{t-2}+2^{t-1}-2=s_{t-1}-2=n-2^t-1$, and we need the inequality $2^{t-1}-2\le n-2^t-1$. Now we have $n-2^t+1=\sum_{j=0}^{t-1}\alpha_j2^j\ge2^{t-1}$, and the proof is completed.
\end{proof}

\subsection{The height of $\widetilde w_2$}

This subsection is devoted to proving Theorem \ref{thm2}. For that purpose we exhibit two crucial relations in $\mathbb Z_2[w_2,w_3]$, which involve the polynomials $g_r$, $r\ge0$, and the ideals $I_n$, $n\ge2$. These relations are obtained in Propositions \ref{prva lema cong} and \ref{druga lema cong}.

We begin with two equalities in $\mathbb Z_2[w_2,w_3]$ involving polynomials $g_r$ only (the ideals $I_n$ will enter the stage afterwards). They are proved in the following two lemmas.

\begin{lemma}\label{prva lema}
The following identity holds in $\mathbb Z_2[w_2,w_3]$ for all $t\ge3$:
\[w_2^{2^{t-2}-2}g_{2^t-5}=w_3^{2^{t-1}-3}+\sum_{i=1}^{t-3}w_2^{2^{t-2}-2^{i+1}}w_3^{2^i-2}g_{2^t+2^i-3}\]
(it is understood that the sum equals zero in the case $t=3$).
\end{lemma}
\begin{proof}
We prove this lemma by induction on $t$. The induction base (the case $t=3$) reduces to $g_3=w_3$, which we know is true (see Table \ref{table:2}).

Proceeding to the induction step, we take $t\ge4$ and suppose that the corres\-ponding equality holds for $t-1$:
\[w_2^{2^{t-3}-2}g_{2^{t-1}-5}=w_3^{2^{t-2}-3}+\sum_{i=1}^{t-4}w_2^{2^{t-3}-2^{i+1}}w_3^{2^i-2}g_{2^{t-1}+2^i-3}.\]
Squaring this identity and then multiplying by $w_3^3$ leads to
\[w_2^{2^{t-2}-4}w_3^3(g_{2^{t-1}-5})^2=w_3^{2^{t-1}-3}+\sum_{i=1}^{t-4}w_2^{2^{t-2}-2^{i+2}}w_3^{2^{i+1}-1}(g_{2^{t-1}+2^i-3})^2.\]
We now use Lemma \ref{kvadriranje} to get
\[
w_2^{2^{t-2}-4}w_3^2g_{2^t-7}=w_3^{2^{t-1}-3}+\sum_{i=1}^{t-4}w_2^{2^{t-2}-2^{i+2}}w_3^{2^{i+1}-2}g_{2^t+2^{i+1}-3}.
\]
According to (\ref{lemspol1}), $w_3^2g_{2^t-7}=w_2^2g_{2^t-5}+g_{2^t-1}$. Using this and shifting the index in the sum, we obtain
\begin{align*}
w_2^{2^{t-2}-4}w_2^2g_{2^t-5}&=w_2^{2^{t-2}-4}g_{2^t-1}+w_3^{2^{t-1}-3}+\sum_{i=2}^{t-3}w_2^{2^{t-2}-2^{i+1}}w_3^{2^i-2}g_{2^t+2^i-3}\\
                             &=w_3^{2^{t-1}-3}+\sum_{i=1}^{t-3}w_2^{2^{t-2}-2^{i+1}}w_3^{2^i-2}g_{2^t+2^i-3},
\end{align*}
and we are done.
\end{proof}

\begin{lemma}\label{druga lema}
For all integers $s$ and $t$ such that $1\le s\le t-1$ the following identity holds in $\mathbb Z_2[w_2,w_3]$:
\[w_2^{2^{s-1}-1}(g_{2^t-2^{s-1}-2})^2=\sum_{j=0}^{s-2}w_2^{2^{s-1}-2^{j+1}}w_3^{2^j-1}g_{2^{t+1}-2^s+2^j-3}\]
(it is understood that the right-hand side equals zero in the case $s=1$).
\end{lemma}
\begin{proof}
 We prove this by induction on $s$. The base case $s=1$ is just Proposition \ref{g-3}(a).
    Now suppose $(s,t)$ is a pair with $2\le s\le t-1$. By inductive hypothesis applied to $(s-1,t-1)$ the following equality holds:
    \[
        w_2^{2^{s-2}-1}(g_{2^{t-1}-2^{s-2}-2})^2=\sum_{j=0}^{s-3}w_2^{2^{s-2}-2^{j+1}}w_3^{2^j-1}g_{2^{t}-2^{s-1}+2^j-3}.
    \]
    After multiplying this equation with $w_3$ and applying Lemma \ref{kvadriranje} to its left-hand side we get:
    \begin{equation}\label{pocetna2}
        w_2^{2^{s-2}-1}g_{2^{t}-2^{s-1}-1}=\sum_{j=0}^{s-3}w_2^{2^{s-2}-2^{j+1}}w_3^{2^j}g_{2^{t}-2^{s-1}+2^j-3}.
    \end{equation}
    Next, we square the equation (\ref{pocetna2}) and multiply it by $w_3$ (in that order):
    \[w_2^{2^{s-1}-2}w_3\left(g_{2^{t}-2^{s-1}-1}\right)^2=\sum_{j=0}^{s-3}w_2^{2^{s-1}-2^{j+2}}w_3^{2^{j+1}}\cdot w_3\left(g_{2^{t}-2^{s-1}+2^j-3}\right)^2.\]
 Now we apply Lemma \ref{kvadriranje} to both sides, and shift the index in the sum:
       \[ w_2^{2^{s-1}-2}g_{2^{t+1}-2^{s}+1}=\sum_{j=1}^{s-2}w_2^{2^{s-1}-2^{j+1}}w_3^{2^{j}}g_{2^{t+1}-2^{s}+2^j-3}.
    \]
According to (\ref{recgpolk3}) and Lemma \ref{kvadriranje}, for the right-hand side we have:
    \begin{align*}
        w_2^{2^{s-1}-2}g_{2^{t+1}-2^{s}+1}&=w_2^{2^{s-1}-2}(w_3g_{2^{t+1}-2^{s}-2}+w_2g_{2^{t+1}-2^{s}-1})\\
                                        &=w_2^{2^{s-1}-2}w_3g_{2^{t+1}-2^{s}-2}+w_2^{2^{s-1}-1}w_3\left(g_{2^{t}-2^{s-1}-2}\right)^2,
    \end{align*}
    which implies
        \begin{align*}
            w_2^{2^{s-1}-1}w_3\left(g_{2^{t}-2^{s-1}-2}\right)^2\!&=w_2^{2^{s-1}-2}w_3g_{2^{t+1}-2^{s}-2}+\!\sum_{j=1}^{s-2}w_2^{2^{s-1}-2^{j+1}}w_3^{2^{j}}g_{2^{t+1}-2^{s}+2^j-3}\\
            &=\sum_{j=0}^{s-2}w_2^{2^{s-1}-2^{j+1}}w_3^{2^{j}}g_{2^{t+1}-2^{s}+2^j-3},
        \end{align*}
    Canceling out $w_3$ concludes the proof.
\end{proof}

Let us now consider the ideals $I_n\trianglelefteq\mathbb Z_2[w_2,w_3]$, $n\ge2$. Recall that these form a descending sequence, and that $I_n$ is generated by the polynomials $g_{n-2}$, $g_{n-1}$ and $g_n$. We will also work with the ideals $w_3I_n=\{w_3p\mid p\in I_n\}$, $n\ge2$. These ideals behave very nicely when it comes to squaring. That property is stated in the following lemma, which will be used extensively in the rest of the section.

\begin{lemma}\label{kvadriranje2}
Let $p\in\mathbb Z_2[w_2,w_3]$ and $n\ge2$. If $p\in w_3I_n$, then $p^2\in w_3I_{2n+1}$. In particular, the following implication holds:
\[p\in w_3I_n\,\Longrightarrow\, p^2\in w_3I_{2n}.\]
\end{lemma}
\begin{proof}
If $p\in w_3I_n$, then $p=w_3(p_{n-2}g_{n-2}+p_{n-1}g_{n-1}+p_ng_n)$ for some polynomials $p_{n-2}$, $p_{n-1}$ and $p_n$. According to Lemma \ref{kvadriranje} we have
\begin{align*}
p^2&=w_3^2(p_{n-2}^2g_{n-2}^2+p_{n-1}^2g_{n-1}^2+p_n^2g_n^2)\\
   &=w_3(p_{n-2}^2g_{2n-1}+p_{n-1}^2g_{2n+1}+p_n^2g_{2n+3})\in w_3I_{2n+1},
\end{align*}
by (\ref{uskladiti}). Since $I_{2n+1}\subseteq I_{2n}$, the second part of the lemma is now obvious.
\end{proof}

We first use Lemma \ref{kvadriranje2} to prove one of two key relations announced at the beginning of this subsection. This relation is a consequence of Lemma \ref{prva lema}.

\begin{proposition}\label{prva lema cong}
For all integers $t\ge3$ we have:
\[w_2^{2^t-4}+w_2^{2^{t-2}-1}w_3^{2^{t-1}-2}\equiv w_2^{2^{t-1}-3}g_{2^t-2}\pmod{w_3I_{2^t-1}}.\]
\end{proposition}
\begin{proof}
The proof is by induction on $t$. We see from Table \ref{table:2} that $w_2^4+w_2w_3^2= w_2g_6$, and so the proposition is true for $t=3$.

Now let $t\ge4$ and assume that the stated relation holds for $t-1$:
\[w_2^{2^{t-1}-4}+w_2^{2^{t-3}-1}w_3^{2^{t-2}-2}\equiv w_2^{2^{t-2}-3}g_{2^{t-1}-2}\pmod{w_3I_{2^{t-1}-1}}.\]
If we square this relation, according to (\ref{kvadrat}) and Lemma \ref{kvadriranje2} (its first part: $p\in w_3I_n\Rightarrow p^2\in w_3I_{2n+1}$), we get
\[
w_2^{2^t-8}+w_2^{2^{t-2}-2}w_3^{2^{t-1}-4}\equiv w_2^{2^{t-1}-6}g_{2^t-4}\pmod{w_3I_{2^t-1}}.
\]
By (\ref{recgpolk3}) we know that $w_2g_{2^t-4}=w_3g_{2^t-5}+g_{2^t-2}$. Using this and multiplying the previous congruence by $w_2^4$ we obtain
\begin{equation}\label{kkong}
w_2^{2^t-4}+w_2^{2^{t-2}+2}w_3^{2^{t-1}-4}\equiv w_2^{2^{t-1}-3}w_3g_{2^t-5}+w_2^{2^{t-1}-3}g_{2^t-2}\pmod{w_3I_{2^t-1}}.
\end{equation}
Now we apply Proposition \ref{g-3}(c) and conclude that
\[w_2^{2^{t-2}+2}w_3^{2^{t-1}-4}=w_2^2w_3^{2^{t-2}-3}w_2^{2^{t-2}}w_3^{2^{t-2}-1}=w_2^2w_3^{2^{t-2}-3}g_{2^t+2^{t-2}-3}\in w_3I_{2^t-1}\]
(by (\ref{uskladiti})). On the other hand, if we use Lemma \ref{prva lema}, we get that
\begin{align*}
w_2^{2^{t-1}-3}w_3g_{2^t-5}&=w_2^{2^{t-2}-1}w_3\bigg(w_3^{2^{t-1}-3}+\sum_{i=1}^{t-3}w_2^{2^{t-2}-2^{i+1}}w_3^{2^i-2}g_{2^t+2^i-3}\bigg)\\
&\equiv w_2^{2^{t-2}-1}w_3^{2^{t-1}-2}\pmod{w_3I_{2^t-1}}
\end{align*}
(again by (\ref{uskladiti})). Therefore, (\ref{kkong}) reduces to
\[w_2^{2^t-4}\equiv w_2^{2^{t-2}-1}w_3^{2^{t-1}-2}+w_2^{2^{t-1}-3}g_{2^t-2}\pmod{w_3I_{2^t-1}},\]
which concludes the induction step.
\end{proof}

The second key relation is straightforward from Lemma \ref{druga lema} (in the sum from that lemma, the only summand which remains is the one for $j=0$; by (\ref{uskladiti}) all other summands belong to the ideal $w_3I_{2^{t+1}-2^s}$).

\begin{proposition}\label{druga lema cong}
For all integers $s$ and $t$ such that $2\le s\le t-1$ we have:
\[w_2^{2^{s-1}-1}(g_{2^t-2^{s-1}-2})^2\equiv w_2^{2^{s-1}-2}g_{2^{t+1}-2^s-2}\pmod{w_3I_{2^{t+1}-2^s}}.\]
\end{proposition}

We now establish an important relation for the proof of Theorem \ref{thm2} in the case $2^t-1\le n\le2^t+2^{t-1}$.

\begin{theorem}\label{prva polovina}
If $t\ge3$ is an integer, then
\[w_2^{2^t-3}\equiv w_2^{2^{t-2}-2}g_{2^t+2^{t-1}-2}\pmod{w_3I_{2^t+2^{t-1}}}.\]
\end{theorem}
\begin{proof}
We prove the theorem by induction on $t$. From Table \ref{table:2} we see that $w_2^5=g_{10}$, and conclude that the relation holds for $t=3$.

Now, for $t\ge4$, assuming
\[w_2^{2^{t-1}-3}\equiv w_2^{2^{t-3}-2}g_{2^{t-1}+2^{t-2}-2}\pmod{w_3I_{2^{t-1}+2^{t-2}}},\]
we use Lemma \ref{kvadriranje2} to obtain
\[w_2^{2^t-6}\equiv w_2^{2^{t-2}-4}(g_{2^{t-1}+2^{t-2}-2})^2\pmod{w_3I_{2^t+2^{t-1}}}.\]
Multiplying this congruence by $w_2^3$ and using Proposition \ref{druga lema cong} (for $s=t-1$) we get
\begin{align*}
w_2^{2^t-3}&\equiv w_2^{2^{t-2}-1}(g_{2^{t-1}+2^{t-2}-2})^2=w_2^{2^{t-2}-1}(g_{2^t-2^{t-2}-2})^2\\
   &\equiv w_2^{2^{t-2}-2}g_{2^{t+1}-2^{t-1}-2}=w_2^{2^{t-2}-2}g_{2^t+2^{t-1}-2}\pmod{w_3I_{2^t+2^{t-1}}},
\end{align*}
and the induction step is completed.
\end{proof}

The following theorem will be essential in determining the height of $\widetilde w_2$ for $n$ in the second half of the interval $[2^t-1,2^{t+1}-1)$.

\begin{theorem}\label{druga polovina}
Let $s$ and $t$ be integers such that $1\le s\le t-2$. Then
\[w_2^{2^{t+1}-3\cdot2^s-1}+w_2^{2^{t-1}-1}w_3^{2^t-2^{s+1}}\equiv w_2^{2^t-2^{s+1}-2^{s-1}}g_{2^{t+1}-2^s-2}\pmod{w_3I_{2^{t+1}-2^s}}.\]
\end{theorem}
\begin{proof}
The proof is by induction on $s$. So, we first establish the relation for $s=1$ (and arbitrary $t\ge3$). We start off by squaring the relation obtained in Proposition \ref{prva lema cong}:
\[w_2^{2^{t+1}-8}+w_2^{2^{t-1}-2}w_3^{2^t-4}\equiv w_2^{2^t-6}(g_{2^t-2})^2\pmod{w_3I_{2^{t+1}-2}}\]
(by Lemma \ref{kvadriranje2}). We know that $(g_{2^t-2})^2=g_{2^{t+1}-4}$ (see (\ref{kvadrat})). Inser\-ting this in the previous congruence and multiplying by $w_2$, we obtain
\[w_2^{2^{t+1}-7}+w_2^{2^{t-1}-1}w_3^{2^t-4}\equiv w_2^{2^t-5}g_{2^{t+1}-4}\pmod{w_3I_{2^{t+1}-2}},\]
and this is the desired relation in the case $s=1$.

Proceeding to the induction step, let $s\ge2$, $t\ge s+2$, and suppose that the theorem is true for the pair of integers $(s-1,t-1)$:
\[w_2^{2^t-3\cdot2^{s-1}-1}+w_2^{2^{t-2}-1}w_3^{2^{t-1}-2^s}\equiv w_2^{2^{t-1}-2^s-2^{s-2}}g_{2^t-2^{s-1}-2}\pmod{w_3I_{2^t-2^{s-1}}}.\]
Similarly as in the induction base, we use Lemma \ref{kvadriranje2} to square this congruence, and then multiply by $w_2$:
\[w_2^{2^{t+1}-3\cdot2^s-1}+w_2^{2^{t-1}-1}w_3^{2^t-2^{s+1}}\equiv w_2^{2^t-2^{s+1}-2^{s-1}+1}(g_{2^t-2^{s-1}-2})^2\pmod{w_3I_{2^{t+1}-2^s}}.\]
Finally, according to Proposition \ref{druga lema cong} we have
\begin{align*}
w_2^{2^t-2^{s+1}-2^{s-1}+1}(g_{2^t-2^{s-1}-2})^2&=w_2^{2^t-2^{s+1}-2^s+2}w_2^{2^{s-1}-1}(g_{2^t-2^{s-1}-2})^2\\
                                              &\equiv w_2^{2^t-2^{s+1}-2^s+2}w_2^{2^{s-1}-2}g_{2^{t+1}-2^s-2}\\
                                              &=w_2^{2^t-2^{s+1}-2^{s-1}}g_{2^{t+1}-2^s-2}\pmod{w_3I_{2^{t+1}-2^s}},
\end{align*}
completing the proof.
\end{proof}

Now we have all that we need for the proof of Theorem \ref{thm2}.

\begin{proof}[Proof of Theorem \ref{thm2}] Let $n\ge7$ and $t\ge3$ be integers such that $2^t-1\le n<2^{t+1}-1$.

\medskip

In the case $2^t-1\le n\le2^t+2^{t-1}$ we need to verify that $\widetilde w_2^{2^t-4}\neq0$ and $\widetilde w_2^{2^t-3}=0$ in $H^*(\widetilde G_{n,3})$, i.e., $w_2^{2^t-4}\notin I_n$ and $w_2^{2^t-3}\in I_n$ (see (\ref{ekv})). Since $I_{2^t-1}\supseteq I_n\supseteq I_{2^t+2^{t-1}}$, it suffices to prove that $w_2^{2^t-4}\notin I_{2^t-1}$ and $w_2^{2^t-3}\in I_{2^t+2^{t-1}}$.

By Proposition \ref{prva lema cong} we have $w_2^{2^t-4}+w_2^{2^{t-2}-1}w_3^{2^{t-1}-2}\in I_{2^t-1}$, i.e.,
\begin{equation}\label{prva lema cong 2}
w_2^{2^t-4}\equiv w_2^{2^{t-2}-1}w_3^{2^{t-1}-2}\pmod{I_{2^t-1}}.
\end{equation}
In Theorem \ref{Grebner} we have a Gr\"obner basis $F=\{f_0,f_1,\ldots,f_{t-1}\}$ for the ideal $I_{2^t-1}$. From Example \ref{2t-1} we see that the monomial $w_2^{2^{t-2}-1}w_3^{2^{t-1}-2}$ is not divisible by $\mathrm{LM}(f_{t-1})$ and $\mathrm{LM}(f_{t-2})$. Since $\deg_{w_2}\big(\mathrm{LM}(f_i)\big)$ decreases with $i$ (see (\ref{LMgiopada})), for $0\le i\le t-3$ we have $\deg_{w_2}\big(\mathrm{LM}(f_i)\big)>\deg_{w_2}\big(\mathrm{LM}(f_{t-2})\big)=2^{t-2}$. This means that $w_2^{2^{t-2}-1}w_3^{2^{t-1}-2}$ is not divisible by any $\mathrm{LM}(f_i)$, $0\le i\le t-1$. Since $F$ is a Gr\"obner basis, we conclude $w_2^{2^{t-2}-1}w_3^{2^{t-1}-2}\notin I_{2^t-1}$, which implies (by (\ref{prva lema cong 2})) that $w_2^{2^t-4}\notin I_{2^t-1}$.

The fact $w_2^{2^t-3}\in I_{2^t+2^{t-1}}$ is immediate from Theorem \ref{prva polovina}.

\medskip

In the case $2^t+2^{t-1}<n<2^{t+1}-1$ let $s\in\{1,2,\ldots,t-2\}$ be the (unique) integer such that $2^{t+1}-2^{s+1}+1\le n\le2^{t+1}-2^s$. We want to show that $\widetilde w_2^{2^{t+1}-3\cdot2^s-1}\neq0$ and $\widetilde w_2^{2^{t+1}-3\cdot2^s}=0$ in $H^*(\widetilde G_{n,3})$. Similarly as in the previous case, since $I_{2^{t+1}-2^{s+1}+1}\supseteq I_n\supseteq I_{2^{t+1}-2^s}$, we actually need to prove that $w_2^{2^{t+1}-3\cdot2^s-1}\notin I_{2^{t+1}-2^{s+1}+1}$ and $w_2^{2^{t+1}-3\cdot2^s}\in I_{2^{t+1}-2^s}$. We do this by using Theorem \ref{druga polovina}. Since $g_{2^{t+1}-2^s-2}\in I_{2^{t+1}-2^s}$ (and of course, $w_3I_{2^{t+1}-2^s}\subseteq I_{2^{t+1}-2^s}$), this theorem implies
\begin{equation}\label{druga lema cong 2}
w_2^{2^{t+1}-3\cdot2^s-1}+w_2^{2^{t-1}-1}w_3^{2^t-2^{s+1}}\in I_{2^{t+1}-2^s}\subseteq I_{2^{t+1}-2^{s+1}+1}.
\end{equation}

From Example \ref{2t+1-2s+1+1} we see that the monomial $w_2^{2^{t-1}-1}w_3^{2^t-2^{s+1}}$ is not divisible by any of the leading monomials $\mathrm{LM}(f_i)$, $0\le i\le t-1$, from the Gr\"obner basis $F$ of the ideal $I_{2^{t+1}-2^{s+1}+1}$ (as in the previous case, $\deg_{w_2}\big(\mathrm{LM}(f_i)\big)>2^{t-1}$ for $0\le i\le t-3$). This means that $w_2^{2^{t-1}-1}w_3^{2^t-2^{s+1}}\notin I_{2^{t+1}-2^{s+1}+1}$, and consequently, $w_2^{2^{t+1}-3\cdot2^s-1}\notin I_{2^{t+1}-2^{s+1}+1}$ (by (\ref{druga lema cong 2})).

In order to prove $w_2^{2^{t+1}-3\cdot2^s}\in I_{2^{t+1}-2^s}$, we multiply (\ref{druga lema cong 2}) by $w_2$ and obtain
\[w_2^{2^{t+1}-3\cdot2^s}\equiv w_2^{2^{t-1}}w_3^{2^t-2^{s+1}}\pmod{I_{2^{t+1}-2^s}}.\]
So it suffices to show that $w_2^{2^{t-1}}w_3^{2^t-2^{s+1}}\in I_{2^{t+1}-2^s}$. By looking at the Gr\"obner basis $F$ for $I_{2^{t+1}-2^s}$ (Example \ref{2t+1-2s}) we see that \[w_2^{2^{t-1}}w_3^{2^t-2^{s+1}}=w_2^{2^{t-1}}w_3^{2^{t-1}-2^s}w_3^{2^{t-1}-2^s}=w_3^{2^{t-1}-2^s}f_{t-2}\in I_{2^{t+1}-2^s},\]
and the proof is complete.
\end{proof}

\section{Cup-length of $\widetilde G_{n,3}$}
\label{cup-length}

A positive dimensional cohomology class is \em indecomposable \em if it cannot be written as a polynomial in classes of smaller dimension. It is clear that the cup-length is reached by a product of indecomposable classes. A well-known fact is that the Grassmannian $\widetilde G_{n,3}$ is simply connected, which implies that the Stiefel--Whitney classes $\widetilde w_2$ and $\widetilde w_3$ are indecomposable in $H^*(\widetilde G_{n,3})$. So $\cupp_{\mathbb Z_2}(\widetilde G_{n,3})$ is reached by a product of the form
\begin{equation}\label{cuplength1}
\widetilde w_2^b\widetilde w_3^cx_1x_2\cdots x_m
\end{equation}
for some nonnegative integers $b$, $c$ and $m$, where $x_1,x_2,\ldots,x_m$ are some indecomposable classes other than $\widetilde w_2$ and $\widetilde w_3$. Let us also note that the dimension of the monomial (\ref{cuplength1}) must be equal to the dimension of the manifold $\widetilde G_{n,3}$, that is $3n-9$. Namely, otherwise, by Poincar\'e duality there would exist a (posi\-tive dimensional) class $y$ such that $\widetilde w_2^b\widetilde w_3^cx_1x_2\cdots x_my\neq0$ in $H^{3n-9}(\widetilde G_{n,3})$, and we would have a longer nontrivial cup product.

We will also need the following well-known fact (see e.g.\ \cite[p.\ 1171]{Korbas:ChRank}):
\begin{equation}\label{topdim}
\widetilde w_2^b\widetilde w_3^c\neq0\mbox{ in } H^*(\widetilde G_{n,3}) \quad \Longrightarrow \quad 2b+3c<3n-9
\end{equation}
(i.e., the nonzero class in $H^{3n-9}(\widetilde G_{n,3})$ is not a polynomial in $\widetilde w_2$ and $\widetilde w_3$). A consequence of (\ref{topdim}) and the preceding discussion is that a monomial of the form $\widetilde w_2^b$ does not realize the cup-length, and so
\[\cupp_{\mathbb Z_2}(\widetilde G_{n,3})>\height(\widetilde w_2).\]

Recall that the \em characteristic rank \em of the canonical bundle $\widetilde\gamma_{n,3}$, denoted by $\mathrm{charrank}(\widetilde{\gamma}_{n,3})$, is the greatest integer $d$ with the property that for all $q\le d$ every cohomology class in $H^q(\widetilde G_{n,3})$ is a polynomial in Stiefel--Whitney classes $\widetilde w_2$ and $\widetilde w_3$ of $\widetilde\gamma_{n,3}$. Put in other words, the smallest dimension containing an indecomposable class other than $\widetilde w_2$ and $\widetilde w_3$ is $1+\mathrm{charrank}(\widetilde{\gamma}_{n,3})$. It is known (see \cite[Theorem 1]{PPR:ChRank} or \cite[Theorem A]{BasuChakraborty}) that if $t\ge3$ is the integer such that $2^t-1\le n<2^{t+1}-1$, then
\[\mathrm{charrank}(\widetilde{\gamma}_{n,3})=\min\{3n-2^{t+1}-2,2^{t+1}-5\}
=\begin{cases}
3n-2^{t+1}-2, &\!\! n<2^t-1+2^t/3\\
2^{t+1}-5, &\!\! n>2^t-1+2^t/3
\end{cases}.\]

The following lemma is now immediate.

\begin{lemma}\label{monom}
Let $n\ge7$ and $t\ge3$ be integers such that $2^t-1\le n<2^{t+1}-1$, let $x\in H^*(\widetilde G_{n,3})$ be a (homogeneous) class that is not a polynomial in $\widetilde w_2$ and $\widetilde w_3$, and let $|x|$ denotes its (cohomological) dimension.
\begin{itemize}
\item[(a)] If $n<2^t-1+2^t/3$, then $|x|\ge3n-2^{t+1}-1$.
\item[(b)] If $n>2^t-1+2^t/3$, then $|x|\ge2^{t+1}-4$.
\end{itemize}
\end{lemma}

In parts (a) and (b) of the next lemma we strengthen the assertion (\ref{topdim}). The part (c) will be used in the proof of Theorem \ref{thm1}. The main point of (c) is the existence of a nonzero monomial \em in cohomological dimension \em $3n-2^{t+1}-5$ (if the stated conditions are satisfied).

\begin{lemma}\label{nonzeromonomials}
Let $n\ge7$ and $t\ge3$ be integers such that $2^t-1\le n<2^{t+1}-1$, and let $\widetilde w_2^b\widetilde w_3^c$ be a nonzero monomial in $H^*(\widetilde G_{n,3})$.
\begin{itemize}
\item[(a)] If $n<2^t-1+2^t/3$, then $2b+3c\le2^{t+1}-8$.
\item[(b)] If $n>2^t-1+2^t/3$, then $2b+3c\le3n-2^{t+1}-5$.
\item[(c)] If $2^t-1+2^t/3<n\le2^t+2^{t-1}$ and $2^{t+1}-8<2b+3c\le3n-2^{t+1}-5$, then there exist nonnegative integers $k$ and $l$ such that
           $\widetilde w_2^{b+k}\widetilde w_3^{c+l}\neq0$ and $2(b+k)+3(c+l)=3n-2^{t+1}-5$.
\end{itemize}
\end{lemma}
\begin{proof}
(a) Due to (\ref{topdim}) we know that $2b+3c<3n-9$, and Poincar\'e duality applies to give us a class $y\in H^*(\widetilde G_{n,3})$ with the property $\widetilde w_2^b\widetilde w_3^cy\neq0$ in $H^{3n-9}(\widetilde G_{n,3})$. Again by (\ref{topdim}) $y$ cannot be a polynomial in $\widetilde w_2$ and $\widetilde w_3$, and so $|y|\ge3n-2^{t+1}-1$ by Lemma \ref{monom}(a). Now we have
\[2b+3c=3n-9-|y|\le3n-9-(3n-2^{t+1}-1)=2^{t+1}-8.\]

(b) This claim is proved by using Lemma \ref{monom}(b) in the same way as part (a).

(c) For a class $y\in H^*(\widetilde G_{n,3})$ such that $\widetilde w_2^b\widetilde w_3^cy\neq0$ in $H^{3n-9}(\widetilde G_{n,3})$ we have
\[2^{t+1}-4\le|y|=3n-9-(2b+3c)<3n-9-(2^{t+1}-8)=3n-2^{t+1}-1.\]
However, according to \cite[Theorem A]{BasuChakraborty}, in this dimension range there is only one indecomposable class $a_{2^{t+1}-4}\in H^{2^{t+1}-4}(\widetilde G_{n,3})$ (up to addition of a polynomial in $\widetilde w_2$ and $\widetilde w_3$). This means that we can take $y$ to be of the form $\widetilde w_2^k\widetilde w_3^la_{2^{t+1}-4}$ (the exponent of $a_{2^{t+1}-4}$ must be at least $1$ due to (\ref{topdim}), and it cannot be $2$ or more because $n\le2^t+2^{t-1}$ implies $2(2^{t+1}-4)>3n-2^{t+1}-1>|y|$).

Finally, as a consequence of $\widetilde w_2^{b+k}\widetilde w_3^{c+l}a_{2^{t+1}-4}=\widetilde w_2^b\widetilde w_3^cy\neq0$ in $H^{3n-9}(\widetilde G_{n,3})$ we have $\widetilde w_2^{b+k}\widetilde w_3^{c+l}\neq0$, and
\[2(b+k)+3(c+l)=3n-9-(2^{t+1}-4)=3n-2^{t+1}-5,\]
which is what we wanted to prove.
\end{proof}

In the following theorem we establish that the subalgebra $W_n\le H^*(\widetilde G_{n,3})$ (gene\-rated by $\widetilde w_2$ and $\widetilde w_3$) completely determines $\cupp_{\mathbb Z_2}(\widetilde G_{n,3})$ (the number $M_n$ from the theorem is actually the cup-length of this subalgebra).

\begin{theorem}\label{thmcuplength}
Let $n\ge7$ be an integer.
\begin{itemize}
\item[(a)] If (\ref{cuplength1}) is a monomial which realizes $\cupp_{\mathbb Z_2}(\widetilde G_{n,3})$, then $m=1$ (that is, there is exactly one indecomposable class other than $\widetilde w_2$ and $\widetilde w_3$ in (\ref{cuplength1})).\\
\item[(b)] If
$M_n=\max\big\{b+c\mid\widetilde w_2^b\widetilde w_3^c\neq0\mbox{ in } H^*(\widetilde G_{n,3})\big\}$, then
\[\cupp_{\mathbb Z_2}(\widetilde G_{n,3})=M_n+1.\]
\end{itemize}
\end{theorem}
\begin{proof}
(a) Let $\widetilde w_2^b\widetilde w_3^cx_1x_2\cdots x_m$ be a monomial that realizes $\cupp_{\mathbb Z_2}(\widetilde G_{n,3})$, i.e., $b+c+m=\cupp_{\mathbb Z_2}(\widetilde G_{n,3})$, where $x_1,x_2,\ldots,x_m$ are (not necessarily mutually different) indecomposable classes other than $\widetilde w_2$ and $\widetilde w_3$. We know that for the dimension of this monomial the following holds:
\begin{equation}\label{dimcuplength}
2b+3c+\sum_{i=1}^m|x_i|=3n-9.
\end{equation}

The inequality $m\ge1$ is immediate from (\ref{topdim}), and so, we are left to prove that $m\le1$. Let $t\ge3$ be the integer such that $2^t-1\le n<2^{t+1}-1$. We will distinguish two cases.

\medskip

\underline{Case 1}: If $2^t-1\le n<2^t-1+2^t/3$, then for all $i$ we have $|x_i|\ge3n-2^{t+1}-1$ (Lemma \ref{monom}(a)). So by (\ref{dimcuplength}) and the inequality $2^t\le n+1$ we get
\[3n-9\ge\sum_{i=1}^m|x_i|\ge m(3n-2^{t+1}-1)\ge m(3n-(2n+2)-1)=m(n-3),\]
and we conclude that $m\le3$. Moreover, $m$ might be equal to $3$ only if $b=c=0$ and $2^{t+1}=2n+2$, i.e., $n=2^t-1$. But that would mean that $\cupp_{\mathbb Z_2}(\widetilde G_{2^t-1,3})=3$, which is not possible, since by Theorem \ref{thm2} we have $\height(\widetilde w_2)=2^t-4>3$.
Therefore, $m\le2$, and it remains to rule out the possibility $m=2$.

Suppose $m=2$. Then by (\ref{dimcuplength}) and the inequality $n\ge2^t-1$ we would have
\begin{align*}
2(b+c)&\le2b+3c=3n-9-|x_1|-|x_2|\le3n-9-2(3n-2^{t+1}-1)\\
      &=2^{t+2}-3n-7\le2^{t+2}-3(2^t-1)-7=2^t-4,
\end{align*}
which would imply $\cupp_{\mathbb Z_2}(\widetilde G_{n,3})=b+c+2\le2^{t-1}-2+2=2^{t-1}$. On the other hand, $\cupp_{\mathbb Z_2}(\widetilde G_{n,3})>\height(\widetilde w_2)=2^t-4$ (Theorem \ref{thm2}), which is clearly a contradiction (because $t\ge3$).

\medskip

\underline{Case 2}: If $2^t-1+2^t/3<n\le2^{t+1}-2$, then $|x_i|\ge2^{t+1}-4$ for all $i$ (Lemma \ref{monom}(b)). Similarly as in the first case, we have
\[3n-9\ge\sum_{i=1}^m|x_i|\ge m(2^{t+1}-4)\ge m(n-2),\]
implying $m\le2$. Then, assuming $m=2$ we get
\[2(b+c)\le2b+3c=3n-9-|x_1|-|x_2|\le3n-9-2(2^{t+1}-4)=3n-2^{t+2}-1.\]

If $n\le2^t+2^{t-1}$, then
\[2(b+c)\le3n-2^{t+2}-1\le3(2^t+2^{t-1})-2^{t+2}-1=2^{t-1}-1.\]
This would mean that $\cupp_{\mathbb Z_2}(\widetilde G_{n,3})=b+c+2\le2^{t-2}-1+2=2^{t-2}+1$. However, we know that $\cupp_{\mathbb Z_2}(\widetilde G_{n,3})>\height(\widetilde w_2)=2^t-4$.

If $2^t+2^{t-1}+1\le n\le2^{t+1}-2$, then
\[2(b+c)\le3n-2^{t+2}-1\le3(2^{t+1}-2)-2^{t+2}-1=2^{t+1}-7,\]
and so $\cupp_{\mathbb Z_2}(\widetilde G_{n,3})=b+c+2\le2^t-4+2=2^t-2$. A contradiction is now derived by the fact $\cupp_{\mathbb Z_2}(\widetilde G_{n,3})>\height(\widetilde w_2)\ge2^t+2^{t-2}-1$ (Theorem \ref{thm2}; Table \ref{table:1}).

\medskip

(b) By the discussion at the beginning of the section, the cup-length is realized by a monomial of the form (\ref{cuplength1}), and by part (a) of the theorem
\[\cupp_{\mathbb Z_2}(\widetilde G_{n,3})=b+c+1\le M_n+1.\]
On the other hand, if $\widetilde w_2^{\overline b}\widetilde w_3^{\overline c}$ is a nonzero monomial with $\overline b+\overline c=M_n$, then (\ref{topdim}) gives $2\overline b+3\overline c<3n-9$, and by Poincar\'e duality, there exists a class $y\in H^{3n-9-2\overline b-3\overline c}(\widetilde G_{n,3})$ with the property $\widetilde w_2^{\overline b}\widetilde w_3^{\overline c}y\neq0$, leading to the conclusion
\[\cupp_{\mathbb Z_2}(\widetilde G_{n,3})\ge\overline b+\overline c+1=M_n+1.\]
This completes the proof of the theorem.
\end{proof}

Note that $\big\{b+c\mid\widetilde w_2^b\widetilde w_3^c\neq0\mbox{ in } H^*(\widetilde G_{n,3})\big\}\subseteq\big\{b+c\mid\widetilde w_2^b\widetilde w_3^c\neq0\mbox{ in } H^*(\widetilde G_{n+1,3})\big\}$. Namely, if $\widetilde w_2^b\widetilde w_3^c\neq0$ in $H^*(\widetilde G_{n,3})$, i.e., $w_2^bw_3^c\notin I_n$, then $w_2^bw_3^c\notin I_{n+1}$ (since $I_{n+1}\subseteq I_n$), i.e., $\widetilde w_2^b\widetilde w_3^c\neq0$ in $H^*(\widetilde G_{n+1,3})$. This means that
\begin{equation}\label{increasing}
M_n\le M_{n+1} \mbox{ for all } n.
\end{equation}

We are finally able to compute the $\mathbb Z_2$-cup-length of $\widetilde G_{n,3}$ for all $n$.

\begin{proof}[Proof of Theorem \ref{thm1}] We distinguish four cases.

\underline{Case $2^t+2^{t-1}<n<2^{t+1}-1$}: Let $s\in\{1,2,\ldots,t-2\}$ be the integer such that $2^{t+1}-2^{s+1}+1\le n\le2^{t+1}-2^s$. We want to show that $\cupp_{\mathbb Z_2}(\widetilde G_{n,3})=n-2^s-1$. According to Theorem \ref{thmcuplength}(b) it suffices to establish that
\[M_n=n-2^s-2.\]

We are going to prove first that $\widetilde w_2^{2^{t+1}-3\cdot2^s-1}\widetilde w_3^{n-2^{t+1}+2^{s+1}-1}\neq0$ in $H^*(\widetilde G_{n,3})$, which will imply $M_n\ge n-2^s-2$.
We will do this by actually proving that the monomial $w_2^{2^{t+1}-3\cdot2^s-1}w_3^{n-2^{t+1}+2^{s+1}-1}$ is not an element of the ideal $I_n$ (see (\ref{ekv})).

In the proof of Theorem \ref{thm2} we have established (see (\ref{druga lema cong 2})) the fact
\[w_2^{2^{t+1}-3\cdot2^s-1}+w_2^{2^{t-1}-1}w_3^{2^t-2^{s+1}}\in I_{2^{t+1}-2^s},\]
and since $I_{2^{t+1}-2^s}\subseteq I_n$, we have
\[w_2^{2^{t+1}-3\cdot2^s-1}\equiv w_2^{2^{t-1}-1}w_3^{2^t-2^{s+1}}\pmod{I_n}.\]
Multiplying this congruence by $w_3^{n-2^{t+1}+2^{s+1}-1}$ we get
\begin{equation}\label{druga lema cong 3}
w_2^{2^{t+1}-3\cdot2^s-1}w_3^{n-2^{t+1}+2^{s+1}-1}\equiv w_2^{2^{t-1}-1}w_3^{n-2^t-1}\pmod{I_n}.
\end{equation}

In order to show that $w_2^{2^{t-1}-1}w_3^{n-2^t-1}\notin I_n$, let us look at the Gr\"obner basis $F$ for $I_n$ in this case. Since $n-2^t+1\ge2^t-2^{s+1}+2\ge2^t-2^{t-1}+2>2^{t-1}$, we have $\alpha_{t-1}=1$, and so $s_{t-2}=n-2^t+1-2^{t-1}$ (see (\ref{fpol})). Now (\ref{ft-1}) gives us $f_{t-1}=w_3^{n-2^t}$, while Proposition \ref{LMgi} implies that
\[\deg_{w_2}\big(\mathrm{LM}(f_{t-2})\big)=\frac{n+1-s_{t-2}}{2}-2^{t-2}=2^{t-1}.\]
Also, $\deg_{w_2}\big(\mathrm{LM}(f_i)\big)>2^{t-1}$ for $0\le i\le t-3$ by (\ref{LMgiopada}). Now we see that $w_2^{2^{t-1}-1}w_3^{n-2^t-1}$ is not divisible by any of the leading monomials from $F$, which means that $w_2^{2^{t-1}-1}w_3^{n-2^t-1}\notin I_n$. By (\ref{druga lema cong 3}), $w_2^{2^{t+1}-3\cdot2^s-1}w_3^{n-2^{t+1}+2^{s+1}-1}\notin I_n$, and we have the inequality $M_n\ge n-2^s-2$.

We are left to prove $M_n\le n-2^s-2$. It suffices to show that $b+c\le n-2^s-2$ for every nonzero monomial $\widetilde w_2^b\widetilde w_3^c$ in $H^*(\widetilde G_{n,3})$. Assume to the contrary that $\widetilde w_2^b\widetilde w_3^c$ is a nonzero monomial with $b+c\ge n-2^s-1$. We know that $b\le2^{t+1}-3\cdot2^s-1$ (Theorem \ref{thm2}), and we get
\[c\ge n-2^s-1-b\ge n-2^s-1-2^{t+1}+3\cdot2^s+1=n-2^{t+1}+2^{s+1}.\]
Therefore,
\[2b+3c=2(b+c)+c\ge2(n-2^s-1)+n-2^{t+1}+2^{s+1}=3n-2^{t+1}-2.\]
However, this contradicts Lemma \ref{nonzeromonomials}(b).

\medskip

\underline{Case $n=2^t+2^{t-1}$}: The claim is that $\cupp_{\mathbb Z_2}(\widetilde G_{2^t+2^{t-1},3})=2^t-1$. As in the previous case, we use Theorem \ref{thmcuplength}(b) to reduce the claim to $M_{2^t+2^{t-1}}=2^t-2$.

Observe the monomial $w_2^{2^{t-1}-1}w_3^{2^{t-1}-1}$. We are going to prove that it is not an element of the ideal $I_{2^t+2^{t-1}}$, i.e., that the cohomology class $\widetilde w_2^{2^{t-1}-1}\widetilde w_3^{2^{t-1}-1}$ is nonzero, and we will have $M_{2^t+2^{t-1}}\ge2^t-2$.

In Example \ref{2t+2t-1} we calculated the last three polynomials from the Gr\"obner basis $F$ of $I_{2^t+2^{t-1}}$. They are:
\[f_{t-3}=w_2^{2^{t-1}+2^{t-3}}w_3^{2^{t-3}-1},\,
f_{t-2}=w_2^{2^{t-1}}w_3^{2^{t-2}-1} \mbox{ and }
f_{t-1}=w_3^{2^{t-1}}.\]
By (\ref{LMgiopada}), the leading monomials of all other polynomials from $F$ have the exponent of $w_2$ greater than $2^{t-1}$. So we see that there is no polynomial $f\in F$ such that $\mathrm{LM}(f)\mid w_2^{2^{t-1}-1}w_3^{2^{t-1}-1}$. Since $F$ is a Gr\"obner basis, $w_2^{2^{t-1}-1}w_3^{2^{t-1}-1}\notin I_{2^t+2^{t-1}}$.

In order to prove $M_{2^t+2^{t-1}}\le2^t-2$ we take a nonzero monomial $\widetilde w_2^b\widetilde w_3^c$ in $H^*(\widetilde G_{2^t+2^{t-1},3})$, and we want to show that $b+c\le2^t-2$. Assume to the contrary that $b+c>2^t-2$. The plan is to show that this would imply existence of a nonzero monomial of this form (with sum of the exponents greater than $2^t-2$) \em in dimension $5(2^{t-1}-1)$\em, and then to provide a contradiction by proving that such a monomial does not exist in that dimension.

Lemma \ref{nonzeromonomials}(b) implies $2b+3c\le3n-2^{t+1}-5$, and since $b+c>2^t-2$ we have
\[2b+3c\ge2(b+c)>2^{t+1}-4.\]
So Lemma \ref{nonzeromonomials}(c) applies to give us a nonzero monomial in $\widetilde w_2$ and $\widetilde w_3$ in dimension $3n-2^{t+1}-5=3(2^t+2^{t-1})-2^{t+1}-5=5(2^{t-1}-1)$,
whose sum of the exponents is greater than $2^t-2$. This monomial must be of the form $\widetilde w_2^{2^{t-1}-1+3j}\widetilde w_3^{2^{t-1}-1-2j}$ for some $j>0$ (the sum of the exponents is $2^t-2+j$).

Now we obtain a contradiction by proving that all (corresponding) monomials $w_2^{2^{t-1}-1+3j}w_3^{2^{t-1}-1-2j}$ with $j>0$ belong to the ideal $I_{2^t+2^{t-1}}$.
If $0<j\le2^{t-3}$, then
\begin{align*}
w_2^{2^{t-1}-1+3j}w_3^{2^{t-1}-1-2j}&=w_2^{3j-1}w_3^{2^{t-2}-2j}w_2^{2^{t-1}}w_3^{2^{t-2}-1}\\
                                    &=w_2^{3j-1}w_3^{2^{t-2}-2j}f_{t-2}\in I_{2^t+2^{t-1}}.
\end{align*}
If $2^{t-3}<j\le2^{t-3}+2^{t-4}$, then
\begin{align*}
w_2^{2^{t-1}-1+3j}w_3^{2^{t-1}-1-2j}&=w_2^{3j-2^{t-3}-1}w_3^{2^{t-2}+2^{t-3}-2j}w_2^{2^{t-1}+2^{t-3}}w_3^{2^{t-3}-1}\\
                                    &=w_2^{3j-2^{t-3}-1}w_3^{2^{t-2}+2^{t-3}-2j}f_{t-3}\in I_{2^t+2^{t-1}}.
\end{align*}
Finally, if $j>2^{t-3}+2^{t-4}$, then $2^{t-1}-1+3j>2^t+2^{t-4}-1>2^t-4$, and $\widetilde w_2^{2^{t-1}-1+3j}=0$ since $\height(\widetilde w_2)=2^t-4$ (Theorem \ref{thm2}).

\medskip

\underline{Case $n=2^t+2^{t-1}-1$}: The proof is similar to the one in the previous case. The monomial $w_2^{2^{t-1}-1}w_3^{2^{t-1}-2}$ is not divisible by any of the leading monomials $\mathrm{LM}(f)$, $f\in F$, (Example \ref{2t+2t-1-1}). This means that $M_{2^t+2^{t-1}-1}\ge2^t-3$.

For the opposite inequality, if $\widetilde w_2^b\widetilde w_3^c\neq0$ in $H^*(\widetilde G_{2^t+2^{t-1}-1,3})$, and $b+c>2^t-3$, then Lemma \ref{nonzeromonomials}(b,c) ensures that there is no loss of generality in assuming
\[2b+3c=3n-2^{t+1}-5=3(2^t+2^{t-1}-1)-2^{t+1}-5=5\cdot2^{t-1}-8.\]
Therefore, this monomial must be of the form $\widetilde w_2^{2^{t-1}-1+3j}\widetilde w_3^{2^{t-1}-2-2j}$ for some $j>0$.

To obtain a contradiction (as in the previous case) we look at the Gr\"obner basis $F$ for $I_{2^t+2^{t-1}-1}$ (Example \ref{2t+2t-1-1}) and prove that $w_2^{2^{t-1}-1+3j}w_3^{2^{t-1}-2-2j}\in I_{2^t+2^{t-1}-1}$ if $j>0$:
\begin{align*}
w_2^{2^{t-1}-1+3j}w_3^{2^{t-1}-2-2j}=\,&w_2^{3j-1}w_3^{2^{t-2}-1-2j}f_{t-2}\in I_{2^t+2^{t-1}-1}\quad (\mbox{if } 0<j\le2^{t-3}-1);\\
w_2^{2^{t-1}-1+3j}w_3^{2^{t-1}-2-2j}=\,&w_2^{3j-2^{t-3}-1}w_3^{2^{t-2}+2^{t-3}-1-2j}f_{t-3}\in I_{2^t+2^{t-1}-1}\\
                                     &(\mbox{if } 2^{t-3}\le j\le2^{t-3}+2^{t-4}-1);
\end{align*}
and if $j>2^{t-3}+2^{t-4}-1$, then $\widetilde w_2^{2^{t-1}-1+3j}\widetilde w_3^{2^{t-1}-2-2j}=0$ because $2^{t-1}-1+3j>2^t+2^{t-4}-4>2^t-4=\height(\widetilde w_2)$.

\medskip

\underline{Case $2^t-1\le n\le2^t+2^{t-1}-2$}: In this final case the statement we need to prove is $\cupp_{\mathbb Z_2}(\widetilde G_{n,3})=2^t-3$, i.e., $M_n=2^t-4$ (Theorem \ref{thmcuplength}(b)). By Theorem \ref{thm2}, $\height(\widetilde w_2)=2^t-4$, and so $M_n\ge2^t-4$.

In order to prove the opposite inequality we first note that $M_n\le M_{2^t+2^{t-1}-2}$ by (\ref{increasing}). This means that it is sufficient to prove $M_{2^t+2^{t-1}-2}\le2^t-4$.

So let $\widetilde w_2^b\widetilde w_3^c$ be a nonzero class in $H^*(\widetilde G_{2^t+2^{t-1}-2,3})$. We want to show that $b+c\le2^t-4$. Assume to the contrary that $b+c>2^t-4$. As before, since $2b+3c\ge2(b+c)>2^{t+1}-8$, we can use Lemma \ref{nonzeromonomials}(b,c) to achieve
\[2b+3c=3n-2^{t+1}-5=3(2^t+2^{t-1}-2)-2^{t+1}-5=5\cdot2^{t-1}-11.\]
By (\ref{increasing}) and the previous case, we also have
\[2^t-4<b+c\le M_{2^t+2^{t-1}-2}\le M_{2^t+2^{t-1}-1}=2^t-3,\,\mbox{ i.e., }\, b+c=2^t-3.\]
The only solution to the system
\begin{align*}
b+c&=2^t-3\\
2b+3c&=5\cdot2^{t-1}-11
\end{align*}
is the pair $(b,c)=(2^{t-1}+2,2^{t-1}-5)$. Therefore, for $t=3$ the specified class does not exist, and we are done. For $t\ge4$, by Proposition \ref{g-3}(d) and (\ref{uskladiti}) we have
\begin{align*}
w_2^{2^{t-1}+2}w_3^{2^{t-1}-5}&=w_2^2w_3^{2^{t-2}-4}w_2^{2^{t-1}}w_3^{2^{t-2}-1}\\
                              &=w_2^2w_3^{2^{t-2}-4}g_{2^t+2^{t-1}+2^{t-2}-3}\in I_{2^t+2^{t-1}-2},
\end{align*}
which contradicts the assumption $\widetilde w_2^b\widetilde w_3^c\neq0$ in $H^*(\widetilde G_{2^t+2^{t-1}-2,3})$. This concludes the proof of Theorem \ref{thm1}.
\end{proof}

\bibliographystyle{amsplain}

\end{document}